\documentclass[a4paper,12pt]{amsart}
\usepackage{amsmath}
\usepackage{amsmath}
\usepackage{amssymb}
\usepackage{amscd}
\usepackage{mathrsfs}
\usepackage[all]{xy}
\usepackage[dvipdfm]{graphicx}
\def\mathcal{\mathscr}
\everymath{\displaystyle}
\newtheorem{thm}{Theorem}[section]
\newtheorem{lem}[thm]{Lemma}
\newtheorem{cor}[thm]{Corollary}
\newtheorem{prop}[thm]{Proposition}

\theoremstyle{definition}

\newtheorem{rem}[thm]{Remark}

\newtheorem{defn}[thm]{Definition}

\newcommand{\mca}[1]{{\mathcal{#1}}}
\newcommand\Q{{\mathbb Q}}
\newcommand\Z{{\mathbb Z}}

\newcommand\R{{\mathbb R}}
\newcommand\capp{\text{\rm cap}}
\newcommand\const{\text{\rm const}}

\newcommand\diam{\text{\rm diam}}
\newcommand\dist{\text{\rm dist}}
\newcommand\ep{\varepsilon}
\newcommand\Hess{\text{\rm Hess}}
\newcommand\HZ{\text{\rm HZ}}
\newcommand\id{\text{\rm id}}
\newcommand\Ima{\text{\rm Im}}
\newcommand\inj{\text{\rm inj}}
\newcommand\interior{\text{\rm int}}
\newcommand\pr{\text{\rm pr}}
\newcommand\St{\text{\rm St}}
\newcommand\supp{\text{\rm supp}}
\newcommand\vol{\text{\rm vol}}
\newcommand\Vo{\text{\rm Vo}}

\begin{document}
\pagestyle{plain}
\thispagestyle{plain}

\title[Displacement energy of unit disk cotangent bundles]
{Displacement energy of unit disk cotangent bundles}

\author[Kei Irie]{Kei Irie}
\address{Research Institute for Mathematical Sciences, Kyoto University, 
Kyoto 606-8502, Japan.}
\email{iriek@kurims.kyoto-u.ac.jp}


\maketitle

\begin{abstract}
We give an upper bound of a Hamiltonian displacement energy of a unit disk cotangent bundle $D^*M$
in a cotangent bundle $T^*M$, when the base manifold $M$ is an open Riemannian manifold. 
Our main result is that the displacement energy is not greater than $C r(M)$, where $r(M)$ is the inner radius of $M$, and $C$ is a dimensional constant. 
As an immediate application, we study symplectic embedding problems of unit disk cotangent bundles. 
Moreover, combined with results in symplectic geometry, our main result shows the 
existence of short periodic billiard trajectories and short geodesic loops. 
\end{abstract}

\section{Introduction}
\subsection{Displacement energy} 

\textit{Displacement energy} is an important quantity in symplectic geometry, introduced by H. Hofer \cite{Hofer}. 
First we recall its definition. 
Let $(X,\omega)$ be a symplectic manifold. 
For any $H \in C^\infty(X)$, 
its \textit{Hamiltonian vector field} $X_H$ is defined as 
$\omega(X_H, \, \cdot \,) = dH(\, \cdot \,)$. 
For any $H \in C_c^\infty([0,1] \times X)$ ($C_c^\infty$ denotes the set of compactly supported smooth functions)
and $0 \le t \le 1$, 
$H_t \in C_c^\infty(X)$ is defined as $H_t(x):=H(t,x)$  
and its \textit{Hofer norm} $\| H\|$ is defined as 
\[
\|H\| := \int_0^1 \max H_t - \min H_t \, dt. 
\]
For any $H \in C_c^\infty([0,1] \times X)$, isotopy $(\varphi_H^t)_{0 \le t \le 1}$ on $X$ is defined by 
\[
\varphi_H^0 = \id_X, \qquad 
\partial_t \varphi_H^t = X_{H_t} (\varphi_H^t). 
\]
For any compact set $K \subset X$, we define 
\[
e(K:X) := \inf \{ \| H\| \mid H \in C_c^\infty([0,1] \times X) ,\quad  \varphi^1_H(K) \cap K = \emptyset \}. 
\]
For any set $Y \subset X$, we define $e(Y:X):= \sup_K e(K:X)$, where $K$ runs over all compact sets contained in $Y$, and call it 
\textit{displacement energy} of $Y$. 

\subsection{Main result}
For any manifold $M$, its cotangent bundle $T^*M$ carries a canonical symplectic form. 
Let $\pi_M: T^*M \to M$ denote the natural projection, and define $\lambda_M \in \Omega^1(T^*M)$ as 
\[
\lambda_M(\xi):= p(d\pi_M(\xi)) \qquad  \bigl( q \in M, \, p \in T_q^*M, \, \xi \in T_{(q,p)}(T^*M) \bigr). 
\]
Then $\omega_M:= d\lambda_M$ is a symplectic form on $T^*M$. 

Let $M$ be a $n$-dimensional Riemannian manifold without boundary.
$D^*M$ denotes the unit disk cotangent bundle of $M$, i.e.
\[
D^*M:= \bigl\{(q,p) \in T^*M \bigm{|} |p| <1 \bigr\}.
\]
Let us set  
$d(M):= e(D^*M: T^*M)$. 
When $M$ is compact, $d(M)=\infty$ 
since the zero-section of $T^*M$ is not displaceable by Hamiltonian diffeomorphisms (see e.g. \cite{LS}). 
In this paper, we try to bound $d(M)$ from above when $M$ is noncompact. 

We define $r(M)$, the \textit{inner radius} of $M$, as follows ($\dist_M$ denotes the distance function with respect to the Riemannian metric on $M$): 
\begin{itemize}
\item For any compact set $K \subset M$, $r_M(K):= \max_{x \in K} \dist_M(x, M \setminus K)$. 
\item $r(M):= \sup_K r_M(K)$, where $K$ runs over all compact sets in $M$. 
\end{itemize}
Here are a few immediate remarks: 
\begin{itemize}
\item 
For any nonempty open set $U$ in $\R^n$ (with the flat metric), 
$r(U)$ is equal to the supremum of radii of balls in $U$.
\item 
When $N$ is a compact Riemannian manifold with boundary and $M= \interior N$, 
$r(M)= \max_{x \in N} \dist_N(x, \partial N)$. 
\item 
When $N$ is a closed Riemannian manifold and $x \in N$, 
$r(N \setminus \{x\}) = \max_{y \in N} \dist_N(x,y)$. 
In particular, $r(N \setminus \{x\}) \le \diam (N)$. 
\end{itemize}

The main result of this paper is the following:

\begin{thm}\label{mainthm}
Let $n$ be an integer, and 
$M$ be a $n$-dimensional noncompact Riemannian manifold without boundary. 
Then $d(M) \le \const_n r(M)$.
\end{thm}
\begin{rem}
The above inequality means that: there exists a positive constant $c$ depending only on $n$, which satisfies 
$d(M) \le cr(M)$ for any $M$. 
\end{rem}

Viterbo \cite{Viterbo} proves the following result, and apply it to prove the existence of a 
short periodic billiard trajectory (Theorem 4.1 in \cite{Viterbo}).

\begin{thm}[Viterbo]\label{thm:viterbo}
Let $U$ be a nonempty open set in $\R^n$, and $i:T^*U \to T^*\R^n$ be the obvious embedding. 
Then $e(i(D^*U): T^*\R^n) \le \const_n \vol(U)^{1/n}$. 
\end{thm}

Theorem \ref{thm:viterbo} easily follows from Theorem \ref{mainthm}, since $r(U) \le \const_n \vol(U)^{1/n}$ and 
$e(i(D^*U):T^*\R^n) \le e(D^*U:T^*U)$. 

\subsection{Notations} 
Before describing applications of our main result, we fix several notations 
which we will use in the rest of the paper. 

Let $N$ be a Riemannian manifold. 
$\inj (N)$ denotes its injective radius. 
For any $p \in N$, 
$\exp_p: \{v \in T_pN \mid |v|< \inj(N)\} \to N$ denotes an exponential map at $p$. 
For any $p,q \in N$ such that $\dist_N(p,q) < \inj(N)$, 
$\overrightarrow{pq} \in T_pN$ is defined as 
$\exp_p(\overrightarrow{pq})=q$.
$\gamma_{pq}$ denotes the shortest geodesic from $p$ to $q$, i.e. 
$\gamma_{pq}(t):=\exp_p(t\overrightarrow{pq})\,(0 \le t \le 1)$. 
When $p \ne q$, we set $e_{pq}:=\overrightarrow{pq}/|\overrightarrow{pq}|$. 

For any $p \in N$ and positive real numbers $a, b$, we set 
\begin{align*}
B_N(x:a)&:= \{y \in N \mid \dist_N(x,y)<a\}, \qquad
B_N(x:a,b):=\{y \in N \mid a< \dist_N(x,y)<b \}, \\
S_N(x:a)&:= \{y \in N \mid \dist_N(x,y)=a\}. 
\end{align*}
When $N=\R^n$, we denote them as $B^n(x:a)$, $B^n(x:a,b)$, $S^{n-1}(x:a)$. 
In particular, we adopt abbreviations 
$B^n(a):=B^n((0,\ldots,0):a)$, 
$S^{n-1}(a):=S^{n-1}((0,\ldots,0):a)$. 

\subsection{Applications} 

\subsubsection{Expanding embeddings and symplectic embeddings} 

Let $U, V$ be open sets in $\R^n$. 
An embedding $f:U \to V$ is called an \textit{expanding embedding} if there holds 
$|df(\xi)| \ge |\xi|$ for any $\xi \in TU$. 

If there exists an expanding embedding $U \to V$, it obviously induces a symplectic embedding 
$D^*U \to D^*V$. 
It is natural to ask to what extent the converse of this is true. 
This question is addressed in \cite{Guth}, and it is shown (see pp. 479 in \cite{Guth}) 
that symplectic embeddings of unit disk cotangent bundles are more flexible than expanding embeddings, 
when $U$ and $V$ are rectangles in $\R^n$. 
On the other hand, the following corollary of Theorem \ref{mainthm} is somewhat in the opposite direction. 
We use the Hofer-Zehnder capacity (which we denote by $c_\HZ$ ) in the proof. For its definition and basic properties, see \cite{HZ} and references therein. 

\begin{cor}\label{cor:embedding}
Let $V$ be an open set in $\R^n$. 
If there exists a symplectic embedding $D^*B^n(1) \to D^*V$, then 
$r(V) \ge \const_n$. Therefore, there exists an isometric embedding 
$B^n(\const_n) \to V$. 
\end{cor}
\begin{proof}
We set $c:= c_\HZ(D^* B^n(1)) >0$. 
If there exists a symplectic embedding $D^*B^n(1) \to D^*V$, we have 
$c \le c_\HZ(D^*V)$. 
On the other hand, the energy-capacity inequality 
(see \cite{HZ} Chapter 5, Theorem 7) and Theorem \ref{mainthm} show 
$c_\HZ(D^*V) \le d(V) \le \const_n r(V)$. 
Thus $c \le \const_n r(V)$. Since $c$ is a dimensional constant, this completes the proof. 
\end{proof} 

\begin{rem}\label{rem:sharpness}
It is clear that $c_\HZ(D^*B^n(r)) = c_\HZ(D^*B^n(1))r = cr$. Thus, for any open set $V \subset \R^n$ we have  
$c r(V)  \le c_\HZ(D^*V)$. Hence the above energy-capacity inequality implies 
$c r(V) \le d(V)$. This shows that Theorem \ref{mainthm} is sharp (up to a dimensional constant) for open sets in $\R^n$. 
\end{rem} 

\subsubsection{Short periodic billiard trajectory}
First we clarify a definition of periodic billiard trajectory.

\begin{defn}\label{defn:billiard}
Let $N$ be a Riemannian manifold, possibly with boundary. 
Then, a \textit{periodic billiard trajectory} on $N$ is a continuous map $\gamma:\R/\Z \to N$, such that there exists 
a finite set $B_\gamma \subset \R/\Z$ with the following properties:
\begin{enumerate}
\item[(1)] On $(\R/\Z) \setminus B_\gamma$, $\gamma$ is smooth and satisfies $\ddot{\gamma} \equiv 0$.
\item[(2)] For any $t \in B_\gamma$, $\gamma(t) \in \partial N$. Moreover, 
$\dot{\gamma}_{\pm}(t):= \lim_{h \to \pm 0} \dot{\gamma}(t+h)$ satisfy 
\[
\dot{\gamma}_+(t) + \dot{\gamma}_-(t) \in T_{\gamma(t)}\partial N, \qquad
\dot{\gamma}_-(t) - \dot{\gamma}_+(t) \in (T_{\gamma(t)}\partial N)^{\perp} \setminus \{0\}.
\]
$B_\gamma$ is called the set of \textit{bounce times}.
\end{enumerate}
\end{defn}
\begin{rem}\label{rem:billiard}
In the above definition, 
a closed geodesic is a periodic billiard trajectory (the set of bounce times is empty). 
\end{rem}

\begin{prop}\label{prop:billiard}
Let $N$ be a $n$-dimensional compact Riemannian manifold with nonempty boundary.
Then, there exists a periodic billiard trajectory on $N$ with at most $n+1$ bounce times and 
of length not greater than $d(\interior N)$.
\end{prop}

Proposition \ref{prop:billiard} is observed in \cite{Viterbo} (proof of Theorem 4.1), 
although it does not contain the estimate of the number of bounce times. 
We can prove Proposition \ref{prop:billiard} in the same way as Theorem 1.2 in \cite{Albers-Mazzucchelli}. 
The arguments in \cite{Albers-Mazzucchelli} are 
based on a version of the energy-capacity inequality, 
and the approximation technique due to \cite{Benci-Giannoni}.

\begin{rem}\label{rem:AM}
Although \cite{Albers-Mazzucchelli} is working on domains in Euclidean space, 
most arguments in the proof of Theorem 1.2 \cite{Albers-Mazzucchelli} work on general compact Riemannian manifolds with no change, 
at least when the potential function (which is denoted by $V$ in \cite{Albers-Mazzucchelli}) is constant. 
The only place we have to slightly change is the proof of Proposition 2.2 \cite{Albers-Mazzucchelli}. 
More precisely, the formula (2.32) in \cite{Albers-Mazzucchelli} should be replaced with the following formula: 
\[
\Hess \,  \mca{L}^{E_\ep}_\ep(\Gamma_\ep,\tau_\ep)[(\Psi_\ep,0), (\Psi_\ep,0)]=A_\ep - B_\ep -\tau_\ep^{-1} \int_0^1 \langle R(\dot{\Gamma_\ep}, \Psi_\ep) \Psi_\ep , \dot{\Gamma}_\ep \rangle \, dt. 
\]
$R$ denotes the curvature tensor, and the other symbols are same as in \cite{Albers-Mazzucchelli}. 
The last term does not appear in the flat case. However, it is uniformly bounded on $\ep$, since 
$\|\dot{\Gamma}_\ep \|_{L^\infty}$ is bounded (see pp. 3295 in \cite{Albers-Mazzucchelli}), and 
$\|\Psi_\ep \|_{L^\infty}$ and $\tau_\ep^{-1}$ are bounded by the assumption. 
Thus this new term does not violate the proof in \cite{Albers-Mazzucchelli}. 
\end{rem}

By Proposition \ref{prop:billiard}, Theorem \ref{mainthm} implies the following corollary:

\begin{cor}\label{cor:billiard}
Let $N$ be a $n$-dimensional compact Riemannian manifold with nonempty boundary.
Then, there exists a periodic billiard trajectory on $N$ with at most $n+1$ bounce times and 
of length not greater than $\const_n \max_{x \in N} \dist(x, \partial N)$. 
\end{cor}

In \cite{Irie}, Corollary \ref{cor:billiard} is proved 
when $N$ is a bounded domain in Euclidean space. 
This result improves a result of Viterbo on short periodic billiard trajectory
(Theorem 4.1 in \cite{Viterbo}). 
See also \cite{AO} for relevant results. 

\subsubsection{Short geodesic loop}
Given a Riemannian manifold $N$, a \textit{geodesic loop} at $x \in N$ means a geodesic
$c:[0,1] \to N$ such that $c(0)=c(1)=x$. 
$c$ can be singular at $x$, i.e. we do not require that $\dot{c}(0) = \dot{c}(1)$. 
A trivial geodesic loop at $x$ means a constant loop at $x$. 

\begin{cor}\label{cor:loop}
For any closed Riemannian manifold $N$ and $x \in N$, there exists a nontrivial geodesic loop at $x$ 
of length not greater than $\const_n \diam(N)$.
\end{cor}

Corollary \ref{cor:loop} is not a new result.
Actually, Rotman \cite{Rotman} proves 
the following stronger result: 

\begin{thm}[Rotman]\label{thm:Rotman}
For any closed Riemannian manifold $N$ and $x \in N$, there exists a nontrivial geodesic loop at $x$ with length not greater than 
$2j \diam(N)$, where $j:=\min\{ i \mid \pi_i(N) \ne 0\}$. 
In particular, the length of the shortest geodesic loop at $x$ is not greater than $2n\diam(N)$.
\end{thm}

Our proof of Corollary \ref{cor:loop} is completely different from the arguments in \cite{Rotman}, 
and makes use of  the following lemma, based on arguments of Mohnke \cite{Mohnke}.
Recall that for any contact manifold $(Y, \lambda)$, its \textit{symplectization} is $Y \times \R_{>0}$ endowed with a $1$-form 
$\tilde{\lambda}(z,r):= r \lambda(z) \,( z \in Z, r \in \R_{>0})$. 

\begin{lem}\label{lem:Mohnke}
Let $(W,d\lambda)$ be an exact symplectic manifold of bounded geometry, and 
$S$ be a closed hypersurface in $W$. 
Suppose that $(S,\lambda)$ is a contact manifold, and there exists an embedding 
$i: S \times (0,1] \to W$ such that $i^* \lambda = \widetilde{\lambda|_S}$, and its image 
$i(S \times (0,1])$ is Hamiltonian displaceable in $(W,d\lambda)$. 
Then, any closed Legendrean on $(S,\lambda)$ admits a Reeb chord $\gamma$ satisfying $\int_\gamma \lambda \le e(i(S \times (0,1]):W)$. 
\end{lem} 

The proof is same as that of Theorem 4 in \cite{Mohnke} and is omitted. 
Notice that we do not need a condition on $\pi_1$ (which is assumed in Theorem 4 \cite{Mohnke}), since we are working on 
an exact symplectic manifold. 

Now we prove Corollary \ref{cor:loop} from Theorem \ref{mainthm} and Lemma \ref{lem:Mohnke}:
\begin{proof}
Let $\rho$ be a nonincreasing smooth function on $[0,\infty)$ such that 
$\rho\equiv 1$ near $0$, $\supp \rho \subset [0,1]$ and $\rho'(t)<0$ when $0<\rho(t)<1$.
For $0<\delta<\inj(N)$, define $\rho_\delta$ by $\rho_\delta(t):=\rho(t/\delta)$. 
We define $V_\delta \in C^\infty(N)$, 
$H_\delta \in C^\infty(T^*N)$, 
$D_\delta \subset T^*N$ by 
\[
V_\delta(q):=\rho_\delta(\dist_N(x,q)), \quad
H_\delta(q,p):=V_\delta(q)+|p|^2/2, \quad
D_\delta:= \{H_\delta \le 1/2\}.
\]

Then, it is easy to see that there exists $\lambda \in T^*N$ which satisfies
$d\lambda=\omega_N$, $\lambda|_{\{p=0\}} \equiv 0$ and the following property:
\begin{quote}
$(\partial D_\delta, \lambda)$ is a contact manifold. Moreover, there exists an embedding 
$i: \partial D_\delta \times (0,1] \to D_\delta$ such that $i^*\lambda= \widetilde{\lambda|_{\partial D_\delta}}$. 
\end{quote}
For an elementary proof of this fact, see Lemma 10 \cite{Irie}. 
(See also Lemma 5.2 \cite{CFP}.)

Then, Lemma \ref{lem:Mohnke} shows that any Legendrean on $(\partial D_\delta, \lambda)$ has 
a Reeb chord $\gamma$ such that $\int_\gamma \lambda \le e(D_\delta: T^*N)$. 
In particular, a Legendrean $\{p=0\} \cap \partial D_\delta$ has a Reeb chord $\gamma_\delta$ 
such that $\int_{\gamma_\delta} \lambda \le e(D_\delta: T^*N)$. 
On the other hand, since $D_\delta \subset D^*(N \setminus \{x\})$, 
Theorem \ref{mainthm} implies that 
$e(D_\delta: T^*N) \le \const_n r(N \setminus \{x\}) \le \const_n \diam(N)$. 

By reparametrizing $\gamma_\delta$, one gets $\Gamma_\delta: [0,T_\delta] \to T^*N$ such that 
$\Gamma_\delta(0), \Gamma_\delta(T_\delta) \in \{p=0\}$, 
$\partial_t\Gamma_\delta = X_{H_\delta}(\Gamma_\delta)$, and 
$\int_{\Gamma_\delta} \lambda_N  = \int_{\Gamma_\delta} \lambda \le \const_n \diam(N)$
(the first equality holds since $d\lambda=d\lambda_N=\omega_N$, and both $\lambda$ and $\lambda_N$ vanish on $\{p=0\}$). 

Then, $q_\delta:= \pi_N(\Gamma_\delta)$ satisfies 
$\partial_t^2q_\delta + \nabla V_\delta(q_\delta)=0$, $q_\delta(0), q_\delta(T_\delta) \in \{V_\delta=1/2\}$.
Since $0< \delta < \inj(N)$, it is easy to show that 
$q_\delta([0,T_\delta])$ is not contained in $B_N(x: \inj(N))$. 
Since $V_\delta \equiv 0$ on $N \setminus B_N(x: \delta)$, 
there exists a geodesic $c_\delta:[0,T'_\delta] \to N \setminus B_N(x: \delta)$
with length not greater than $\const_n \diam (N)$, $c_\delta(0), c_\delta(T'_\delta) \in S_N(x:\delta)$ and 
$c_\delta([0,T'_\delta])$ is not contained in $B_N(x: \inj(N))$. 

Finally, take an arbitrary sequence $(\delta_n)_n$ such that $\lim_{n \to \infty} \delta_n=0$.
Then, a certain subsequence of 
$(c_{\delta_n})_n$ converges to a nontrivial geodesic loop at $x$ with length not greater than 
$\const_n \diam(N)$.
\end{proof}

\subsection{Relations of Theorem \ref{mainthm}  to \cite{Irie}, \cite{Irie_JSG}, and sharpness of Theorem \ref{mainthm}}

In this subsection, we discuss relations of Theorem \ref{mainthm} to results in \cite{Irie}, \cite{Irie_JSG}. 
We also discuss some speculations on sharpness of Theorem \ref{mainthm}. 
This subsection is less self-contained than the other parts of this paper, and results in this subsection 
are not used in the rest of this paper. 

In \cite{Irie}, the author introduced the notion of capacity of Riemannian manifolds, which we denote by $\capp_R$. 
Let us briefly recall the definition. Let $M$ be a (open) Riemannian manifold without boundary. 
$\mca{V}(M)$ denotes the set of $V \in C^\infty(M)$, such that $0$ is a regular value and $\{V \le 0\}$ is compact. 
For any $V \in \mca{V}(M)$, we set $D_V:= \{ (q,p) \in T^*M \mid V(q)+|p|^2/2 \le 0\}$. Then we define
\[
\capp_R(M):= \sup \{ \capp_S(D_V, \omega_M) \mid V \in \mca{V}(M), V > -1/2 \}. 
\]
For definition of $\capp_S$, see Definition 2 \cite{Irie}.
If $D$ is a restricted contact type domain in $T^*\R^n$, $\capp_S(D, \omega_{\R^n})$ is equal to the Floer-Hofer capacity of $U$ 
(see Definition 5.6, Proposition 5.7 \cite{Hermann}). 

A key result in \cite{Irie} (Theorem 12 \cite{Irie}) claims that, any nonempty open set $U \subset \R^n$ satisfies 
the following estimate, where $c_0(n)$ and $c_1(n)$ are dimensional constants:
\begin{equation}\label{eq:c01}
c_0(n) \le \capp_R(U) / r(U) \le c_1(n).
\end{equation}
In \cite{Irie}, we use this upper bound of $\capp_R(U)$ to prove Corollary \ref{cor:billiard} for bounded domains in $\R^n$.
The next Proposition \ref{prop:cR_disp} shows that, 
$\capp_R$ gives a lower bound of the displacement energy. 

Hence Theorem \ref{mainthm} strengthens and generalizes the upper bound in (\ref{eq:c01}), which is the harder part of (\ref{eq:c01}). 

\begin{prop}\label{prop:cR_disp}
Let $M$ be a $n$-dimensional noncompact Riemannian manifold without boundary. 
Then there holds $\capp_R(M) \le d(M)$. 
\end{prop} 
\textit{Sketch of proof \,}
It is enough to prove $\capp_S(D_V) \le d(M)$ for any $V \in \mca{V}(M)$ such that $V>-1/2$. 
It is easy to see that $D_V \subset D^*M$, thus $e(D_V: T^*M) \le d(M)$. 
Thus it is enough to show that 
$\capp_S(D_V, \omega_M) \le e(D_V: T^*M)$. 
This is a version of the energy-capacity inequality, and when $M$ is an open set in $\R^n$, this follows from 
Theorem 1.4 \cite{Hermann}, which bounds the Floer-Hofer capacity by the displacement energy. 
The proof of this result (see Section 5.3 \cite{Hermann}) 
applies with no change to the general case. 
\qed 

In Theorem 1.6 of our recent paper \cite{Irie_JSG}, we proved that the following refinement of (1) holds for any nonempty open set $U \subset \R^n$:
\begin{equation}\label{eq:c01'}
2 \le \capp_R(U)/ r(U) \le 2(n+1). 
\end{equation}
In \cite{Irie_JSG}, the assumption that $U$ is an open set in $\R^n$ is required for purely technical reasons, while in \cite{Irie} it is a crucial assumption. 
Therefore, it would be reasonable to expect that (\ref{eq:c01'}) holds if we replace $U$ with any open Riemannian manifold $M$.  
If this is true, the expected lower bound implies $2r(M) \le \capp_R(M) \le d(M)$. 
It indicates that 
Theorem \ref{mainthm} is sharp (up to a dimensional constant)  for any open Riemannian manifold.

\section{Width of Riemannian manifolds}

In this section, we introduce a new notion of \textit{width} of Riemannian manifolds. 

In Section 2.1, we explain the definition of the invariant, and prove some of its properties. 
In particular, we bound $d(M)$ by the width of $M$ (Lemma \ref{lem:w}), and reduce Theorem \ref{mainthm} to 
Theorem \ref{thm:width_radius}, which bounds the width from above by the inradius. 
Thus our goal is to prove Theorem \ref{thm:width_radius}. 

In Section 2.2, we prove Theorem \ref{thm:width_radius} 
for open sets in Euclidean space, since this case is considerably simpler than the general case. 

It is far more difficult to  prove Theorem \ref{thm:width_radius} for arbitrary Riemannian manifolds. 
In Section 2.3, we reduce it to Theorem \ref{thm:closed}, which is simpler than Theorem \ref{thm:width_radius}. 
The rest of this paper is devoted to the proof of Theorem \ref{thm:closed}. 
In Section 2.4, we sketch its proof and explain the structure of this paper after Section 2.

\subsection{A definition and properties} 

First we explain a definition of width of Riemannian manifolds. 

\begin{defn} 
Let $M$ be a Riemannian manifold without boundary. 
For any $h \in C^\infty(M)$, we set $\|h\|:= \sup h - \inf h$. 
For any compact set $K$ on $M$, we define  
\begin{align*}
w_M(K):&= \inf \bigl\{ \|h\| \bigm{|}  h \in C^\infty(M),\, \text{$|dh| \ge 1$ on $K$} \bigr\} \\
      &=  \inf \bigl\{ \|h\| \bigm{|}  h \in C_c^\infty(M),\, \text{$|dh| \ge 1$ on $K$} \bigr\}. 
\end{align*}
The \textit{width} of $M$, denoted as $w(M)$, is defined as 
\[
w(M):= \sup_{K \subset M} w_M(K), 
\]
where $K$ runs over all compact sets on $M$. 
\end{defn} 
Let us examine a few simple examples. 
\begin{itemize}
\item When $M$ is a closed Riemannian manifold, $w(M)= w_M(M)=\infty$. 
\item When $M$ is an open interval $(a,b)$ with the standard metric on $\R$, it is easy to see that $w(M)= b-a$. 
\item More generally, when $M: =\{(x_1,\ldots,x_n) \in \R^n \mid a<x_1<b\}$, one can show that 
$w(M)= b-a$. ($w(M) \le b-a$ is easy. The opposite follows from Proposition \ref{prop:width_radius} below. )
\end{itemize} 

The next lemma is immediate from the definition of $w$. 

\begin{lem}
Let $M$ be a Riemannian manifold without boundary, and $U$ be an open set in $M$. 
Then, for any compact set $K$ on $U$, 
$w_U(K)= w_M(K)$. 
In particular, $w(U) \le w(M)$. 
\end{lem} 

\begin{defn} 
For Riemannian manifolds $M, M'$ and a smooth map $\varphi: M \to M'$, we define 
\[
d^+(\varphi):= \sup \{ |d\varphi(\xi)| \mid \xi \in TM, |\xi|= 1 \}, \quad
d^-(\varphi):= \inf \{ |d\varphi(\xi)| \mid \xi \in TM, |\xi|=1 \}.
\]
When we need to specify Riemannian metrics, we denote as $d^{\pm}(\varphi:g_M, g_{M'})$, where $g_M$, $g_{M'}$ are Riemannian metrics on $M$, $M'$. 
\end{defn} 

\begin{lem}\label{lem:gg'}
For any Riemannian manifolds $M$, $M'$ and a diffeomorphism $\varphi: M \to M'$, 
there holds $w(M') \le d^+(\varphi) w(M)$. 
\end{lem}
\begin{proof}
We may assume $w(M), d^+(\varphi)< \infty$. 
It is enough to show that 
for any compact set $K$ on $M'$ and $\varepsilon>0$, there holds 
$w_{M'}(K) \le d^+(\varphi) (w_M(\varphi^{-1}(K))+ \varepsilon)$.

By definition of $w_M(\varphi^{-1}(K))$, there exists $h \in C_c^\infty(M)$ such that 
$|dh| \ge 1$ on $\varphi^{-1}(K)$ and $\|h\|  \le w_M(\varphi^{-1}(K)) + \varepsilon$. 
Then, $g:=d^+(\varphi) \varphi_*h \in C_c^\infty(M')$ satisfies 
$|dg| \ge 1$ on $K$ and $\| g \| = d^+(\varphi) \|h \| \le d^+(\varphi) (w_M(\varphi^{-1}(K))+\varepsilon)$. 
Hence we get $w_{M'}(K) \le d^+(\varphi) (w_M(\varphi^{-1}(K))+ \varepsilon)$. 
\end{proof}

\begin{rem}
There are various ways of measuring the "size" of Riemannian manifolds. 
Relations between these invariants, as well as many applications, are 
extensively studied (see Section 9 in \cite{Guth_ICM} and referneces therein). 
The notion of width $w$, which we introduced above, is a version of such invariants. 
In particular, Lemma \ref{lem:gg'} implies that $w$ is a size invariant in the sense of \cite{Guth_ICM}. 

However, the author could not find any invariant in the literature, which is similar to our invariant $w$. 
A particular property of our invariant $w$ is that it becomes infinity (thus trivial) for closed Riemannian manifolds, therefore it is nontrivial 
only for open Riemannian manifolds. 
\end{rem}

The following simple observation is the first key step in the proof of Theorem \ref{mainthm}:

\begin{lem}\label{lem:w}
Let $M$ be a Riemannian manifold without boudary. Then, $d(M) \le 2w(M)$.
\end{lem}
\begin{proof}
We may assume $w(M)<\infty$.
Let $K$ be a compact set in $D^*M$.
For any $\delta>0$, there exists $h \in C^\infty_0 (M)$ such that 
$\|h\| \le w(M)+\delta$ and $|dh| \ge 1$ on $\pi_M(K)$, where $\pi_M$ is the natural projection $T^*M \to M$.
Let $H:=h \circ \pi \in C^\infty(T^*M)$.
Let $(q_1,\ldots,q_n)$ be a local chart on $M$, and let $(p_1,\ldots,p_n)$ be the associated chart on cotangent fibers.
Then, the Hamiltonian vector field $X_H$ of $H$ is caliculated as:
\[
X_H= \sum_i \frac{\partial H}{\partial q_i} \frac{\partial}{\partial p_i} - \frac{\partial H}{\partial p_i} \frac{\partial}{\partial q_i}
= \sum_i \frac{\partial h}{\partial q_i} \frac{\partial}{\partial p_i}
= dh.
\]
Since $|dh| \ge 1$ on $\pi_M(K)$, and $K \subset D^*M$, $2H$ displaces $K$. 
Although $2H$ is not compactly supported, 
the $1$-parameter group $(\varphi^t_{2H})_{t \in \R}$ of $X_{2H}$ is well-defined, and 
$\bigcup_{0 \le t \le 1}\varphi_{2H}^t(K)$ is compact. 
Hence $e(K:T^*M)$ is bounded by $\| 2H \| = \| 2h \|$, 
therefore by $2w(M)$.
\end{proof}

By Lemma \ref{lem:w}, Theorem \ref{mainthm} reduces to the following result. 

\begin{thm}\label{thm:width_radius}
Let $M$ be a $n$-dimensional noncompact Riemannian manifold without boundary. 
Then, $w(M) \le \const_n r(M)$. 
\end{thm} 

The following simple proposition shows that the above estimate is sharp up to a dimensional constant.  

\begin{prop}\label{prop:width_radius}
Let $M$ be as in Theorem \ref{thm:width_radius}. Then, $w(M) \ge 2r(M)$. 
\end{prop}
\begin{proof}
For any $r$ such that $0< r< r(M)$, 
there exists a compact set $K \subset M$ and $x \in K$ such that $\dist(x, M \setminus K) >r$. 
Let $h \in C_c^\infty(M)$ such that $|dh| \ge 1$ on $K$, and let 
$\gamma:\R \to M$ be the integral curve of $\nabla h$ such that $\gamma(0)=x$. 
Let $[t_0,t_1]$ be the largest interval such that $0 \in [t_0,t_1]$ and $\gamma([t_0,t_1]) \subset K$. 
Then, the length  of $\gamma|_{[t_0,t_1]}$ is larger than $2r$, therefore $h(\gamma(t_1)) - h(\gamma(t_0)) > 2r$. 
This shows $w_M(K) \ge 2r$, thus we have $w(M) \ge 2 r(M)$. 
\end{proof}

\subsection{Proof of Theorem \ref{mainthm} for open sets in $\R^n$}

As a warm up, we prove Theorem \ref{thm:width_radius} (thus Theorem \ref{mainthm}) 
for open sets in $\R^n$, since this case is considerably simpler than the general case. 
First we need the following lemma: 

\begin{lem}\label{lem:Zn}
Let $\Z^n$ be the set of integer points on $\R^n$. Then $w(\R^n \setminus \Z^n) \le n$. 
\end{lem}
\begin{proof}
Consider a continuous function $h_0:\R \to [0,1]$ defined by 
\[
h_0(t):= \begin{cases}
         t-2n &(2n \le t \le 2n+1) \\
         2n+2-t    &(2n+1 \le t \le 2n+2) 
         \end{cases}
\]
where $n$ is an integer. By modifying $h_0$, for any $\delta>0$ one can construct a smooth function 
$h_\delta: \R \to [0,1]$ such that $|h'_\delta(t)| \ge 1$ on $I_\delta:=\bigcup_{n \in \Z} [n+\delta, n+1-\delta]$. 
Define $H_\delta: \R^n \to [0,n]$ by $H_\delta(x_1,\ldots,x_n):=h_\delta(x_1)+\cdots+h_\delta(x_n)$. 

Let $K$ be a compact set in $\R^n \setminus \Z^n$. 
Take any $\delta \in \bigl(0, \dist(K,\Z^n)/\sqrt{n} \bigr)$. 
Then, for any $x=(x_1,\ldots,x_n) \in K$, there exists at least one $1 \le i \le n$ such that $x_i \in I_\delta$. Hence 
\[
|\nabla H_\delta(x)| = \sqrt{\sum_i |h'_\delta(x_i)|^2} \ge 1. 
\]
Hence we get $w_{\R^n \setminus \Z^n}(K)= w_{\R^n} (K) \le n$. 
Thus $w(\R^n \setminus \Z^n) = \sup_K w_{\R^n \setminus \Z^n}(K) \le n$. 
\end{proof}

Now we prove Theorem \ref{thm:width_radius} for open sets in $\R^n$. 

\begin{prop}\label{prop:euclid}
For any open set $U$ in $\R^n$, there holds $w(U) \le \const_n r(U)$. 
\end{prop}
\begin{proof}
We may assume $r(U)<\infty$. 
We abbreviate $r(U)$ as $r$. 
For any $j=(j_1,\ldots,j_n) \in \Z^n$, we set 
$P_j:=(10rj_1, \ldots, 10rj_n)$. 
$B^n(P_j: 2r)$ is not contained in $U$, since $r$ is the supremum of radii of balls in $U$. 
Take $Q_j \in B^n(P_j: 2r) \setminus U$ arbitrary, and set $\mca{Q}:= \{Q_j \mid j \in \Z^n \}$.  
Then obviously $U \subset \R^n \setminus \mca{Q}$, hence $w(U) \le w(\R^n \setminus \mca{Q})$. 

On the other hand, for any $j \in \Z^n$, there exists a compactly supported diffeomorphism 
$\varphi_j: B^n(P_j: 3r) \setminus \{P_j\} \to B^n(P_j: 3r) \setminus \{Q_j\}$, satisfying 
$d^+\varphi_j \le \const_n$. 
Since $B^n(P_j: 3r)$ are disjoint for all $j \in \Z^n$, there exists a diffeomorphism 
$\varphi: \R^n \setminus \Z^n \to \R^n \setminus \mca{Q}$ satisfying $d^+\varphi \le \const_n r$. 
Finally, we obtain 
\[
w(\R^n \setminus \mca{Q}) \le \const_n r \cdot w(\R^n \setminus \Z^n) \le \const_n r, 
\]
where the first inequality follows from Lemma \ref{lem:gg'}, and the second inequality follows from Lemma \ref{lem:Zn}.
\end{proof} 

The above argument appeared in \cite{Irie}. 
It seems difficult to extend this argument for general Riemannian manifolds. 

\subsection{Simplification of Theorem \ref{thm:width_radius}}

As a first step to prove Theorem \ref{thm:width_radius}, we use a doubling argument to reduce it to the following Theorem \ref{thm:closed}, 
which is simpler than Theorem \ref{thm:width_radius}.

\begin{thm}\label{thm:closed}
Any closed, connected $n$-dimensional Riemannian manifold $N$ and a compact set $K \subsetneq N$ satisfy 
$w_N(K) \le \const_n \diam(N)$.
\end{thm}

\begin{proof}[\textbf{\textit{Proof of Theorem \ref{thm:width_radius} modulo Theorem \ref{thm:closed}}}]
Let $M$ be a $n$-dimensional noncompact Riemannian manifold without boundary.
Our aim is to show $w(M) \le \const_n r(M)$, i.e.
any compact set $K$ on $M$ satisfies $w_M(K) \le \const_n r(M)$.
We may assume that $M$ is connected (the general case follows at once from this case).
The key step in the proof is the following claim:
\begin{quote}
For any $\varepsilon>0$, there exists an open neighborhood $W$ of $K$ in $M$, 
a closed, connected Riemannian manifold $N$, and an isometric embedding $i:W \to N$ such that
$\diam (N) \le (2+\varepsilon) r(M)$.
\end{quote}
Once the above claim is established, then we can complete the proof by
\[
w_M(K)= w_N \bigl(i(K)\bigr) \le \const_n \diam (N) \le \const_n r(M), 
\]
where we use Theorem \ref{thm:closed} in the first inequality. 

We prove the claim. 
It is easy to show that there exists $M'$, 
a connected compact submanifold of $M$ with boundary, 
satisfying $K \subset \interior M'$. 

Take an embedding $E: \partial M' \times [-2,1] \to M'$ such that
$E(x,1)=x$, $\Ima E \cap K = \emptyset$.

We set $W:= M' \setminus \Ima E$ and 
$N:= M' \times \{0\} \bigcup_\varphi M' \times \{1\}$, where $\varphi$ is defined by 
\[
\varphi: E(\partial M' \times [-1,1])  \times \{0\} \to E(\partial M' \times [-1,1]) \times \{1\}: \quad 
\bigl(E(u,t), 0 \bigr) \mapsto \bigl(E(u,-t), 1 \bigr).
\]
Then, $N$ is connected.
Moreover we define an embedding $j_0, j_1: M' \to N$ by 
\[
j_k: M' \cong M' \times \{k\} \hookrightarrow N \quad (k=0,1).
\]

For any $\varepsilon>0$, there exists a Riemannian metric $g_N$ on $N$ such that:
\begin{enumerate}
\item[(1)] $d^+(j_0:g_M|_{M'}, g_N) \le 1$. 
\item[(2)] $\diam (M', j_1^* g_N) \le \varepsilon r(M)$.
\item[(3)] $j_0^*g_N|_W \equiv g_M|_W$.
\end{enumerate}
Since $\sup_{x \in M} \dist_M(x, M \setminus M') \le r(M)$, (1) implies 
$\sup_{x \in N} \dist_N(x, \Ima j_1 ) \le r(M)$.
Then (2) implies $\diam (N) \le (2+\varepsilon) r(M)$.
Finally (3) implies that $j_0|_W:W \to N$ is an isometric embedding.
\end{proof}

\subsection{Sketch of a proof of Theorem \ref{thm:closed}}

The rest of this paper is devoted to a proof of Theorem \ref{thm:closed}. 
In this subsection, we sketch the proof, and explain the structure of the rest of this paper. 

For a Riemannian manifold $N$ and $\varepsilon>0$, we introduce a notion of 
\textit{good triangulation} of $N$ with respect to $\varepsilon$. 
Roughly speaking, a triangulation on $N$ is good with respect to $\varepsilon$ if each simplex has a diamter of order $\varepsilon$, 
and its thickness is bounded from below. A rigorous definition will be given in Section 3 (Definition \ref{defn:triangulation}). 

In Section 4, we prove that any Riemannian manifold $N$ has a good triangulation with respect to $\varepsilon$, when 
$0< \varepsilon < \varepsilon(N)$ ($\varepsilon (N)$ is a certain small number which depends on $N$).  
The proof is based on the notion of \textit{Delaunay triangulation}. 
Section 4 is independent from the other parts of this paper, and can be skipped at the first reading. 

In Section 5 and 6, we carry out main part of the proof of Theorem \ref{thm:closed}. 
It is based on the following two propositions:

\textbf{Proposition A:} 
Let $N$ be a $n$-dimensional closed Riemannian manifold, $T$ be a good triangulation of $N$ with respect to $\varepsilon>0$, and 
$V(T) \subset N$ denote the set of vertices of $T$. Then there holds
$w(N \setminus V(T)) \le \const_n \varepsilon$. 

\textbf{Proposition B:} 
For any $n$-dimensional connected closed Riemannian manifold $N$, there exists $\varepsilon(N)>0$ which satisfies the following property:
\begin{quote}
Let $T$ be a good triangulation of $N$ with respect to $\varepsilon \in (0, \varepsilon(N))$. 
Then, for any $v \in V(T)$, there exists an open set $W \subset N \setminus V(T)$ and a diffeomorphism $\Phi: W \to N \setminus \{v\}$ such that 
$d^+(\Phi) \le \const_n \diam(N) \varepsilon^{-1}$. 
\end{quote}

Proposition A is proved in Section 5, Proposition B is proved in Section 6. 
Once we have established Proposition A and B, proof of Theorem \ref{thm:closed} is almost immediate 
(details of this step are spelled out at the beginning of Section 6). 

\section{Good triangulation: definition}

We start with a review of the notion of triangulation.
A simplicial complex $X$ is a pair $\bigl(V(X), \Sigma(X)\bigr)$ where $V(X)$ is a set,  
$\Sigma(X) \subset \{\text{finite subsets of $V(X)$}\}$,
such that 
\begin{itemize}
\item For any $v \in V(X)$, $\{v\} \in \Sigma(X)$.
\item $\tau \subset \sigma, \sigma \in \Sigma(X) \implies \tau \in \Sigma(X)$.
\end{itemize}
For each $\sigma \in \Sigma(X)$, we define $\dim \sigma$ as $\dim \sigma:=\sharp \sigma - 1$, where $\sharp$ denotes the cardinality. 
For each integer $k \ge 0$, we define 
$\Sigma_k(X):=\{ \sigma \in \Sigma(X) \mid \dim \sigma =k\}$.
Moreover, we define a simplicial complex $X_k$ by 
$V(X_k)=V(X)$, $\Sigma(X_k):=\bigcup_{0 \le l \le k} \Sigma_l(X)$.
For each $v \in V(X)$, we define $N_X(v) \subset V(X)$ as 
\[
N_X(v):= \bigl\{ w \in V(X) \bigm{|} \{v,w\} \in \Sigma_1(X) \bigr\}.
\]
An element of $\Sigma(X)$ is called a \textit{simplex} of $X$.
$V(X), \Sigma(X)$ are sometimes abbreviated as $V, \Sigma$.

For any simplex $\sigma=\{v_0, \ldots, v_k\}$, 
we define $|\sigma|, \interior |\sigma| \subset \R[V]:= \bigoplus_{v \in V} \R \cdot v$ by
\[
|\sigma| := \Biggl\{ \sum_{0 \le j \le k} t_jv_j \Biggm{|} 0 \le t_j \le 1, \sum_j t_j= 1 \Biggr\}, \quad
\interior |\sigma|:= \Biggl\{ \sum_{0 \le j \le k} t_jv_j \Biggm{|} 0< t_j  \le 1, \sum_j t_j=1 \Biggr\}.
\]
We set $|X|:= \bigcup_{\sigma \in \Sigma} |\sigma| \subset \R[V]$.
For any $x \in |X|$, we define $\St_X(x) \subset |X|$ by 
$\St_X(x):= \bigcup_{x \in |\sigma|} \interior |\sigma|$. 
$\overline{\St}_X(x):=\bigcup_{x \in |\sigma|} |\sigma|$ is its closure. 

For any $\sigma \in \Sigma(X)$, we equip $|\sigma|$ with a restriction of the standard Riemannian metric on $\bigoplus_{v \in \sigma} \R \cdot v$. 
In this way, we can define a Riemannian metric on $|X|$, and we call it the \textit{standard metric} on $|X|$. 
This metric defines lengths of piecewise smooth paths on $|X|$. 
We also define a distance function on $|X|$ so that $\dist(x,y)$ to be the infimum of lengths of piecewise smooth paths 
connecting $x,y \in |X|$. 

We introduce some terminologies, following \cite{Munkres} Section 8.

\begin{defn}
Let $X$ be a simplicial complex, $N$ be a manifold, and $F:|X| \to N$.
\begin{enumerate}
\item[(1)] $F$ is of \textit{$C^r$-class} if $F|_{|\sigma|}$ is of $C^r$-class for any $\sigma \in \Sigma(X)$.
\item[(2)] When $F$ is of $C^1$-class, $F$ is \textit{nondegenerate} 
if, for any $\sigma \in \Sigma(X)$ satisfying $\dim \sigma \ge 1$,
$d(F|_{|\sigma|})$ has rank equal to $\dim \sigma$ everywhere on $|\sigma|$.
\item[(3)] When $F$ is of $C^1$-class, $(dF)_x: \overline{\St}_X(x) \to T_{F(x)}N$ is defined for each $x \in |X|$ by 
\[
(dF)_x(y):= (dF|_{|\sigma|})_x (y-x), 
\]
where $\sigma \in \Sigma(X)$ such that $x \in |\sigma|$,
and $y$ is an arbitrary point on $|\sigma|$.

\item[(4)] $F$ is a $C^r$-immersion if it is of $C^r$-class and $(dF)_x: \overline{\St}_X(x) \to T_{F(x)}N$ is injective for any $x \in |X|$.
\item[(5)] $(X,F)$ is a $C^r$-triangulation of $N$, if $F$ is a $C^r$-immersion and a homeomorphism.
\end{enumerate}
\end{defn}
\begin{rem}\label{rem:munkres}
Assume that $F:|X| \to N$ is of $C^1$-class.
It is known that $(X,F)$ is a triangulation if 
$F$ is a nondegenerate homeomorphism (Theorem 8.4 in \cite{Munkres}).
\end{rem}

Now we explain the definition of good triangulation. 

\begin{defn}\label{defn:triangulation}
Let $N$ be a $n$-dimensional closed Riemannian manifold.
A triangulation 
$(X,F)$ of $N$ is said to be \textit{good} with respect to $\varepsilon, c_0, c_1, c_2, c_3>0$, 
if it satisfies the following properties: 
\begin{enumerate}
\item[(1)] For any $\sigma \in \Sigma_1(X)$, $F|_{|\sigma|}:|\sigma| \to N$ is a geodesic.
\item[(2)] For any $\sigma \in \Sigma(X)$, $d^+(F|_{|\sigma|}) \le c_1 \varepsilon$ and
$d^-(F|_{|\sigma|}) \ge c_0\varepsilon$, where $|\sigma|$ is equipped with the standard metric. 
\item[(3)] For any $u \in V(X)$, $\sharp N_X(u) \le c_2$.
\item[(4)] For any $u \in V(X)$ and $v,w \in N_X(u)$ such that $v \ne w$, let $\theta$ denote the angle 
between $\overrightarrow{F(u)F(v)}$ and $\overrightarrow{F(u)F(w)}$. Then $\tan (\theta/2) > c_3$. 
\end{enumerate}
\end{defn}
\begin{rem}
It is possible to show that if a triangulation satisfies the above condition (1) and (2), it also satisfies
(3) and (4) with respect to some $c_2$ and $c_3$, which depend on $c_0$ and $c_1$. 
However, it is simpler to verify conditions (3), (4) directly from the construction which we will give in Section 4. 
\end{rem}

The following lemma asserts the existence of a good triangulation. 

\begin{lem}\label{lem:triangulation}
For each integer $n$, there exist constants $c_0, c_1, c_2, c_3$ depending only on $n$, such that the following statement holds:
\begin{quote}
Let $N$ be a $n$-dimensional closed Riemannian manifold. 
When $\varepsilon$ is sufficiently small (i.e. $0<\varepsilon < \varepsilon(N)$), there exists a good triangulation of $N$ 
with respect to $\varepsilon, c_0, c_1, c_2, c_3$. 
\end{quote} 
\end{lem}

Lemma \ref{lem:triangulation} is proved in Section 4. 

\section{Good triangulation: existence}

This section is devoted to the proof of Lemma \ref{lem:triangulation}.
Our idea of proof is to use the notion of \textit{Delaunay triangulation}, 
which is well-known in computational geometry
(see e.g \cite{BCKO}). 

In this section, we define (and prove the existence of) Delaunay triangulation of Riemannian manifolds for point sets satisfying appropriate 
nondegenerate conditions. 
Section 4.1 concerns the case when the manifold is the Euclidean space. 
In Section 4.2, we spell out the nondegenerate conditions on general Riemannian manifolds. 
In Section 4.3, we define Delaunay triangulation on arbitrary closed Riemannian manifolds, and prove Lemma \ref{lem:triangulation} as an 
immediate consequence.  

\subsection{Delaunay triangulation on the Euclidean space}

First we introduce some conditions on subsets of $\R^n$.
Let $S$ be a subset of $\R^n$, and $a,b,c,d$ be positive real numbers.
We define conditions $P_1(a)$, $P_2(b)$, $P_3(c)$, $P_4(d)$ for $S$ as follows:
\begin{itemize}
\item $S$ satisfies $P_1(a)$ $\iff$ Any $s,t \in S, s \ne t$ satisfy $|s-t| \ge a$. 
\item $S$ satisfies $P_2(b)$ $\iff$ Any $x \in \R^n$ satisfies $B^n(x:b) \cap S \ne \emptyset$. 
\item $S$ satisfies $P_3(c)$ $\iff$ Any $x \in \R^n$ and $0<r<c$ satisfy $\sharp(S^{n-1}(x:r) \cap S) \le n+1$.
\item $S$ satisfies $P_4(d)$ $\iff$ Any $s_0, \ldots, s_n \in S$ satisfying 
$\diam \{s_0,\ldots,s_n\} \le d$ are nondegenerate (i.e. there exists no hyperplane in $\R^n$ which contains
$s_0, \ldots, s_n$).
\end{itemize}

For any $s \in S$, define $\Vo(s) \subset \R^n$ by 
\[
\Vo(s):= \bigl\{ x \in \R^n \bigm{|} \text{$|x-s| \le |x-s'|$ for any $s' \in S$} \bigr\}.
\] 
It is called a \textit{Volonoi domain}. 
We define a simplicial complex $X_S$ by $V(X_S)=S$, 
\[
\Sigma(X_S):= \bigl\{ \{s_0,\ldots,s_k\} \subset S \bigm{|} \Vo(s_0) \cap \cdots \cap \Vo(s_k) \ne \emptyset \bigr\}.
\]
Moreover, define $F_S:|X_S| \to \R^n$ so that 
$F_S|_{|\sigma|}$ is an affine map for any $\sigma \in \Sigma(X_S)$.
When $(X_S, F_S)$ is a triangulation of $\R^n$, we call it a \textit{Delaunay triangulation} for $S$. 

\begin{thm}\label{thm:Delaunay}
If $S \subset \R^n$ satisfies $P_1(a)$, $P_2(b)$, $P_3(c)$, $P_4(d)$ and 
$c>b$, $d>2b$, then $(X_S,F_S)$ is a triangulation of $\R^n$.
\end{thm}
\begin{proof}
By Remark \ref{rem:munkres}, it is enough to show that $F_S$ is nondegenerate and 
homeomorphism.
First we show the nondegeneracy. 
If $s,t \in S$ satisfies $\Vo(s) \cap \Vo(t) \ne \emptyset$, then $P_2(b)$ implies that 
$|s-t| \le 2b$. 
Hence any $\sigma \in \Sigma_n(X_S)$ satisfies $\diam(\sigma) \le 2b$. 
Since $d>2b$ and 
$S$ satisfies $P_4(d)$, $\sigma$ is nondegenerate, hence 
$F_S|_{|\sigma|}$ is nondegenerate. 

We have to show that $F_S$ is a homeomorphism.
Since $F_S$ is clearly continuous, it is enough to show its properness, injectivity, and surjectivity.

We show the properness of $F_S$.
Let $K$ be a compact set on $\R^n$, 
and assume that $\sigma \in \Sigma_n(X_S)$ 
satisfies $F_S(|\sigma|) \cap K \ne \emptyset$.
Then, $\diam(\sigma) \le 2b$ implies that 
$\sigma \subset B^n(K:2b):= \bigcup_{x \in K} B^n(x:2b)$.
Hence $F_S^{-1}(K)$ is contained in
$\bigcup_{\sigma \subset B^n(K:2b)} |\sigma|$.
On the other hand, since $S$ satisfies $P_1(a)$, 
$S \cap B^n(K:2b)$ is finite.
Hence $\bigcup_{\sigma \subset B^n(K:2b)} |\sigma|$ is compact.

We show the injectivity of $F_S$.
Arguing indirectly, suppose that there exist $x,y \in |X_S|$ such that $x \ne y$ and 
$F_S(x)=F_S(y)$.
Take $\sigma, \tau \in \Sigma_n(X_S)$ such that $x \in |\sigma|$, $y \in |\tau|$.
Since $F_S|_{|\sigma|}$, $F_S|_{|\tau|}$ are injective,
$\sigma \ne \tau$.
Define $p,q \in \R^n$ and $d_p, d_q >0$ by 
$\{p\}:=\bigcap_{s \in \sigma} \Vo(s)$, $d_p:= \dist(p,S)$, 
$\{q\}:=\bigcap_{t \in \tau} \Vo(t)$, $d_q:=\dist(q,S)$. 
Since $S$ satisfies $P_2(b)$, we get $d_p, d_q \le b$.
Then, since $c>b$ and 
$S$ satisfies $P_3(c)$, 
we get $\sigma = S \cap S^{n-1}(p:d_p)$, $\tau= S \cap S^{n-1}(q:d_q)$.

If $S^{n-1}(p:d_p) \cap S^{n-1}(q:d_q)=\emptyset$, 
then $F_S\bigl(|\sigma|\bigr) \cap F_S\bigl(|\tau|\bigr)=\emptyset$: a 
contradiction.
If $S^{n-1}(p:d_p) \cap S^{n-1}(q:d_q) \ne \emptyset$, one can easily show that 
$x,y \in |\sigma \cap \tau|$. 
However it contradicts the fact that $F_S|_{|\sigma|}, F_S|_{|\tau|}$ are injective.

Finally we show the surjectivity, i.e. $F_S(|X_S|) = \R^n$.
First we claim that for any $x \in \bigcup_{\dim \sigma \ge n-1} \interior |\sigma|$, $F_S(|X_S|)$ is a neighborhood 
of $F_S(x)$.
Let $\sigma$ be a unique simplex of $X_S$ such that $x \in \interior |\sigma|$. 
If $\dim \sigma =n$, then the claim is obvious since $F_S|_{|\sigma|}$ is nondegenerate.
If $\dim \sigma =n-1$, there exists a unique hyperplane $\pi \subset \R^n$, which contains $\sigma$.
$\pi$ divides $\R^n$ into two halfspaces $H_1, H_2$, and 
there exist $s_1 \in H_1 \cap S$, $s_2 \in H_2 \cap S$ such that
$\sigma \cup \{s_1\}, \sigma \cup \{s_2\} \in \Sigma_n(X_S)$.
Then $F_S\bigl(|\sigma \cup \{s_1\}| \cup |\sigma \cup \{s_2\}|\bigr)$ is a neighborhood of 
$F_S(x)$.

Assume that $F_S(|X_S|) \subsetneq \R^n$, and take $x \in \R^n \setminus F_S(|X_S|)$. 
Then, there exists $e \in S^{n-1}$ such that
$(x+e\R) \cap F_S(|X_S|) \ne \emptyset$, and 
$(x+e\R) \cap \bigcup_{\dim \sigma \le n-2} F_S(|\sigma|) = \emptyset$.
Then, the above claim shows that 
$T=\{t \in \R \mid x+et \in F_S(|X_S|) \}$ is open in $\R$.
Moreover, the properness of $F_S$ shows that $T$ is closed.
Finally, obviously $T \ne \emptyset$, $0 \notin T$: a contradiction.
\end{proof}

\subsection{Nondegenerate conditions}
Our idea to prove Lemma \ref{lem:triangulation} is to generalize the notion of Delaunay triangulation for 
finite sets on closed Riemannian manifolds. 
To carry out this idea, 
we consider finite sets on manifolds which satisfy appropriate nondegenerate conditions. 
In this subsection, we spell out those nondegenerate conditions, and show 
the existence of finite sets on manifolds  satisfying those conditions (Lemma \ref{lem:S}).

Let $N$ be a closed Riemannian manifold. 
For $k=1,\ldots,n$ and distinct points 
$x_0, \ldots, x_k \in N$ satisfying  
$\diam \{x_0,\ldots,x_k\} < \inj(N)$, we define 
$\theta_k(x_0,\ldots,x_k)$ as follows  
(recall that we have set in Section 1.3 $e_{pq}:=\overrightarrow{pq}/|\overrightarrow{pq}|$): 
\[
\theta_k(x_0, \ldots, x_k):=\inf_ \sigma 
\vol(e_{x_{\sigma(0)}x_{\sigma(1)}}, \ldots, e_{x_{\sigma(0)}x_{\sigma(k)}}),
\]
where $\sigma$ runs over all permutations of $\{0,\ldots,k\}$, and
\[
\vol(v_1,\ldots,v_k):= \sqrt{\det(v_i \cdot v_j)_{1 \le i,j \le k}} .
\]
Let $S$ be a finite subset of $N$, and 
$a, b, c, d, \delta, \theta$ are positive numbers, and $k=2,\ldots,n$:
\begin{itemize}
\item $S$ satisfies $P'_1(a)$ $\iff$ Any $s,t \in S$, $s \ne t$ satisfy $\dist_N(s,t) \ge a$.
\item $S$ satisfies $P'_2(b)$ $\iff$ Any $x \in N$ satisfies $B_N(x:b) \cap S \ne \emptyset$.
\item $S$ satisfies $P'_3(c,\delta)$ $\iff$ Any $x \in N$ and $0< r < c$ satisfy
$\sharp \bigl( S \cap B_N(x: (1-\delta)r, (1+\delta)r) \bigr) \le n+1$.
\item $S$ satisfies $P'_4(k,d,\theta)$ $\iff$ Any 
$s_0, \ldots, s_k \in S$ with $\diam \{s_0, \ldots,s_k\} \le d$ satisfies 
$\theta_k(s_0, \ldots,s_k) \ge \theta$.
\end{itemize}

\begin{lem}\label{lem:S}
There exist positive numbers $\delta,\theta$ depending only on $n$, which satisfy the followng: 
if $\varepsilon>0$ is sufficiently small, then there exists a finite set $S \subset N$ which satisfies 
$P'_1(\varepsilon/2)$, $P'_2(2\varepsilon)$, $P'_3(5\varepsilon, \delta)$, $P'_4(n,10\varepsilon,\theta)$.
\end{lem}
\begin{rem}
We abbreviate 
$P'_1(\varepsilon/2) \wedge P'_2(2\varepsilon) \wedge P'_3(5\varepsilon,\delta)\wedge P'_4(n,10\varepsilon,\theta)$ 
as $P'(\varepsilon, \delta,\theta)$. 
\end{rem}
\begin{proof}
We prove the following slightly stronger result:
\begin{quote}
There exist positive numbers $\delta,\theta_2, \ldots, \theta_n$ depending only on $n$, which satisfy the following: 
if $\varepsilon>0$ is sufficiently small, then there exists a finite set $S \subset N$ which satisfies 
$P'_1(\varepsilon/2)$, $P'_2(2\varepsilon)$, $P'_3(5\varepsilon, \delta)$, 
$P'_4(k,10\varepsilon,\theta_k)\,(k=2,\ldots,n)$.
\end{quote}

First note that if $S \subset N$ is a maximal set which satisfies $P'_1(\varepsilon)$, then 
it also satisfies $P'_2(\varepsilon)$.
Hence, for any $\varepsilon>0$, there exists $S \subset N$ which satisfies $P'_1(\varepsilon)$, $P'_2(\varepsilon)$.

Take $S=\{s_1,\ldots,s_m\} \subset N$ so that it satisfies $P'_1(\varepsilon)$, $P'_2(\varepsilon)$. 
Then, any $S'=\{s'_1,\ldots,s'_m\} \subset N$ such that 
$s'_i \in B_N(s_i: \varepsilon/10)\,(i=1,\ldots,m)$ 
satisfies $P'_1(4\varepsilon/5)$, $P'_2(11\varepsilon/10)$
(hence it obviously satisfies $P'_1(\varepsilon/2)$, $P'_2(2\varepsilon)$).

We show that there exist $\delta, \theta_2, \ldots, \theta_n$, such that 
for sufficiently small $\varepsilon>0$
we can take $S'$ so that it satisfies $P'_3(5\varepsilon, \delta)$, 
$P'_4(2,10\varepsilon,\theta_2), \ldots, P'_4(n,10\varepsilon,\theta_n)$.

We construct $S'$ inductively. 
First, let $s'_1:=s_1$. 
Suppose that we have chosen $s'_1, \ldots, s'_l$ so that 
$\{s'_1,\ldots,s'_l\}$ satisfies $P'_3(5\varepsilon, \delta), P'_4(2,10\varepsilon,\theta_2), \ldots, P'_4(n,10\varepsilon,\theta_n)$.

We define $W_2,\ldots,W_n, Z \subset B_N(s_{l+1}:\varepsilon/10)$ as follows:
\begin{itemize}
\item $W_k$ is the set of $x \in B_N(s_{l+1}:\varepsilon/10)$ such that
$\{s'_1, \ldots, s'_l, x\}$ does not satisfy $P'_4(k,10\varepsilon, \theta_k)$.
\item $Z$ is the set of $x \in B_N(s_{l+1}:\varepsilon/10)$ such that 
$\{s'_1, \ldots, s'_l, x\}$ does not satisfy $P'_3(5\varepsilon, \delta)$.
\end{itemize}
Then, $\vol(W_2), \ldots, \vol(W_n), \vol(Z)$ are estimated as follows: 
\begin{itemize}
\item For any $c>0$, there exists $\Theta_2>0$ (depending on $c$) and $E>0$ (depending on $c, \Theta_2, N$) 
such that: if $\theta_2 \le \Theta_2$ and $\varepsilon \le E$, then 
$\vol(W_2)/\vol\bigl(B_N(s_{l+1}:\varepsilon/10)\bigr) \le c$.
\item For any $c>0$ and $k=3,\ldots,n$, there exists $\Theta_k>0$ (depending on $c, \theta_{k-1}$)
and $E'>0$ (depending on $c, \theta_{k-1}, \Theta_k, N$) 
such that: 
if $\theta_k \le \Theta_k$ and $\varepsilon \le E'$, then 
$\vol(W_k)/\vol\bigl(B_N(s_{l+1}:\varepsilon/10)\bigr) \le c$.
\item For any $c>0$, there exists $\Delta>0$ (depending on $c$, $\theta_n$)
and $E''>0$ (depending on $c$, $\theta_n$, $\Delta$, $N$)
such that:
if $\delta \le \Delta$ and $\varepsilon \le E''$, then
$\vol(Z)/\vol\bigl(B_N(s_{l+1}:\varepsilon/10)\bigr) \le c$.
\end{itemize}
Therefore, when $\theta_2, \ldots, \theta_n, \delta$ are properly chosen and $\varepsilon>0$ is sufficiently small, 
then $\vol(W_2 \cup \cdots \cup W_n \cup Z)/\vol\bigl(B(s_{l+1}:\varepsilon/10)\bigr) <1$. 
Therefore we can take $s'_{l+1} \in B_N(s_{l+1}:\varepsilon/10)$ such that 
$\{s'_1, \ldots, s'_{l+1}\}$ satisfies 
$P'_3(5\varepsilon, \delta)$,  
$P'_4(2,10\varepsilon,\theta_2), \ldots, P'_4(n,10\varepsilon,\theta_n)$.
Continuing this process until $l+1=m$, we can construct $S'$ which satisfies
$P'(\varepsilon, \delta, \theta_2, \ldots, \theta_n)$.
\end{proof}

\subsection{Delaunay triangulations on Riemannian manifolds}
In this subsection, 
we define Delaunay triangulation 
for nondegenerate finite sets on arbitrary Riemannian manifolds, 
and prove Lemma \ref{lem:triangulation} as an immediate consequence. 

Let $\varepsilon, \delta, \theta$ be positive numbers as in Lemma \ref{lem:S}, 
$N$ be a closed Riemannian manifold, and 
$S$ be a finite set on $N$ which satisfies $P'(\varepsilon, \delta, \theta)$. 
First we construct a simplicial complex $X_S$ and a smooth map $F_S: |X_S| \to N$, and then prove that 
$(X_S,F_S)$ is a good triangulation, when $\varepsilon$ is sufficiently small (Theorem \ref{thm:Delaunay2}).

$X_S$ is defined in exactly the same way as in the case of the Euclidean space (Section 4.1).
For each $s \in S$, we define $\Vo(s) \subset N$ by 
\[
\Vo(s):= \bigl\{ x \in N \bigm{|} \dist_N(s,x) = \dist_N(S,x) \bigr\}.
\]
Then, we define a simplicial complex $X_S$ by $V(X_S)=S$, and 
\[
\Sigma(X_S):= \bigl\{ \{s_0, \ldots, s_k\} \subset S \bigm{|} \Vo(s_0) \cap \cdots \cap \Vo(s_k) \ne \emptyset \bigr\}.
\]
Since $S$ satisfies $P'_2(2\varepsilon)$, 
any $\sigma \in \Sigma(X_S)$ satisfies 
$\diam (\sigma) \le 4\varepsilon$.

Next we define $F_S:|X_S| \to N$. The definition consists of three steps. 

\textbf{Step 1:}
For any $s \in S$, there exists a unique map $i_s: \overline{\St}_{X_S}(s) \to T_sN$ such that 
\begin{itemize}
\item $i_s(t) = \overrightarrow{st}$ for any $t \in N_{X_S}(s)$.
\item For any $\sigma \in \Sigma(X_S)$ such that $s \in \sigma$, $i_s|_{|\sigma|}$ is an affine map.
\end{itemize}
We define $F_s: \overline{\St}_{X_S}(s) \to N$ by 
$F_s:= \exp_s \circ i_s$ .

\textbf{Step 2:}
For each $k \ge 1$ we define 
\[
\mu_k: \bigl\{(q_0,\ldots,q_k)\in N^{k+1} \bigm{|} \diam\{q_0,\ldots,q_k\} < \inj(N) \bigr\} \times \Delta^k \to N,
\]
where 
\[
\Delta^k:= \bigl\{(t_0,\ldots,t_k) \bigm{|} t_0, \ldots, t_k \ge 0, t_0+\cdots+t_k=1 \bigr\}.
\]
Take an arbitrary function $\rho:[0,1] \to [0,1]$, such that 
$\rho \equiv 0$ on some neighborhood of $0$, and 
$\rho \equiv 1$ on some neighborhood of $1$.

Recall that for any points $p,q \in N$ satisfying $\dist_N(p,q)<\inj(N)$, $\gamma_{pq}$ denotes the shortest geodesic from $p$ to $q$ (see Section 1.3). 
When $k=1$, we define $\mu_1$ by 
\[
\mu_1\bigl(q_0, q_1, (t_0, t_1)\bigr) := \gamma_{q_0 q_1}\bigl(\rho(t_1)\bigr).
\]
Suppose that we have defined $\mu_1, \ldots, \mu_{k-1}$. Then, we define $\mu_k$ by 
\[
\mu_k\bigl(q_0, q_1, \ldots, q_k, (t_0,\ldots,t_k)\bigr):=
\begin{cases}
&q_k \qquad\qquad\qquad\qquad\qquad\quad(t_0=\cdots=t_{k-1}=0, t_k=1)\\
&\mu_1(\mu_{k-1}(q_0,\ldots,q_{k-1},\tfrac{t_0}{1-t_k}, \ldots,\tfrac{t_{k-1}}{1-t_k}), t_k) 
\quad (\text{otherwise})
\end{cases}.
\]
As is clear from the definition, $\mu_k$ is $C^\infty$.

\textbf{Step 3:}
Fix an arbitrary order on $S$. 
Then, for any $\sigma=\{s_0, \ldots, s_k\} \in \Sigma(S)$, where  
$s_0 < \cdots < s_k$, we define 
$F_S|_{|\sigma|}$ by 
\[
F_S(x):= \mu_k\bigl(F_{s_0}(x), \ldots, F_{s_k}(x), (t_0, \ldots, t_k) \bigr) \qquad
(x=s_0t_0 + \cdots + s_k t_k, \, t_0 + \cdots + t_k=1 \bigr).
\]
This completes the definition of $F_S:|X_S| \to N$.
When $(X_S, F_S)$ is a triangulation of $N$, we call it a \textit{Delaunay triangulation} for $S$. 

We need the following lemma for later arguments:

\begin{lem}\label{lem:later}
\begin{enumerate}
\item[(1)] $F_S|_{|\sigma|}$ is a geodesic for any $\sigma \in \Sigma_1(X_S)$.
\item[(2)] When $\varepsilon$ is sufficiently small, $F_S(|\sigma|) \subset B_N(\sigma,8\varepsilon)$ for any $\sigma \in \Sigma(X_S)$.
\end{enumerate}
\end{lem}
\begin{proof}
(1):
Set $\sigma:=\{p,q\}$. Then, for $x=(1-\lambda) p+ \lambda q \in |\sigma|$, 
$F_p(x)=F_q(x)=\gamma_{pq}(\lambda)$. 
Therefore 
$F_S\bigl((1-\lambda) p+\lambda q\bigr)=\gamma_{pq}(\lambda)$, hence 
$F_S|_{|\sigma|}$ is a geodesic.

(2):
Set $\sigma:=\{s_0,\ldots,s_k\}$. 
Since $\diam(\sigma) \le 4\varepsilon$, $F_{s_i}(|\sigma|) \subset B_N(s_i,4\varepsilon)$ for each $i=0,\ldots,k$.
Therefore $F_{s_i}(|\sigma|) \subset B_N(s_0,8\varepsilon)$ for each $i$.
On the other hand, when $\varepsilon$ is sufficiently small, $B_N(x,8\varepsilon)$ is geodesically convex for
any $x \in N$.
Hence $F_S(|\sigma|) \subset B_N(s_0,8\varepsilon) \subset B_N(\sigma,8\varepsilon)$.
\end{proof}

In next Theorem \ref{thm:Delaunay2}, we prove that when $\varepsilon$ is sufficiently small, $(X_S,F_S)$ is a good triangulation. 
Obviously, Lemma \ref{lem:triangulation} follows from this result. 

\begin{thm}\label{thm:Delaunay2}
Let $\delta, \theta$ be positive numbers as in Lemma \ref{lem:S}.
There exist positive constants $c_0, c_1, c_2, c_3$ depending only on $n$, such that 
the following holds:  
\begin{quote}
Let $N$ be a $n$-dimensional closed Riemannian manifold. 
When $\varepsilon$ is sufficiently small, 
for any $S \subset N$ satisfying $P'(\varepsilon, \delta, \theta)$, 
$(X_S, F_S)$ is a good triangulation of $N$ with respect to $\varepsilon, c_0, c_1, c_2, c_3$. 
\end{quote}
\end{thm}

To prove Theorem \ref{thm:Delaunay2}, 
first we need the following definition:

\begin{defn}\label{def:converge}
Let $(X_i)_{i=1,2,\ldots}$ be a sequence of subsets of $\R^n$,  
$X_\infty$ be a subset of $\R^n$.
\begin{enumerate}
\item[(1)] When $X_\infty$ is a finite set, $(X_i)_i$ \textit{converges} to $X_\infty$ if 
$\sharp X_i = \sharp X_\infty (:=m)$ for sufficiently large $i$, and one can set 
$X_i=\{x^i_1, \ldots, x^i_m\}$, $X_\infty= \{x^\infty_1, \ldots, x^\infty_m\}$ so that 
$\lim_{i \to \infty} x^i_k = x^\infty_k$ for each $k=1,\ldots,m$. 
\item[(2)] When $X_\infty \cap B^n(r)$ is a finite set for any $r>0$, $(X_i)_i$ \textit{converges} to $X_\infty$ 
if there exists an increasing sequence $(r_j)_j$ of positive real numbers such that 
$\lim_{j \to \infty} r_j= \infty$ and $(X_i \cap B^n(r_j))_i$ converges to $X_\infty \cap B^n(r_j)$ for any $j$.
\end{enumerate}
\end{defn}

\begin{lem}\label{lem:scaling}
Let $(\varepsilon_i)_i$ be a sequence of positive numbers, 
$(S_i)_i$ be a sequence of finite sets on $N$, such that 
$\lim_{i \to \infty} \varepsilon_i=0$ and each 
$S_i$ satisfies $P'(\varepsilon_i, \delta, \theta)$.
Let $(p_i)_i$ be a sequence of points on $N$, 
and $(U_i,\varphi_i)$ be a local chart on $N$ around $p_i$ 
(i.e. $U_i$ is an open neighborhood of $p_i$ and $\varphi:U_i \to \R^n$ is an open embedding), 
with the following conditions: 
\begin{itemize}
\item $\varphi_i(U_i)= \bigl\{x=(x_1,\ldots,x_n) \in \R^n \bigm{|} |x| < \inj(N) \bigr\}$.
\item $\varphi_i(p_i)=(0,\ldots,0)$.
\item $(x_1,\ldots,x_n)$ is a geodesic coordinate on $N$. 
\end{itemize}
Let us define $T_i \subset \R^n$ by 
$T_i:=\varphi_i(U_i \cap S_i)/\varepsilon_i$.
Then, the following holds:
\begin{enumerate}
\item[(1)] Up to subsequence, $(T_i)_i$ converges to some $T_\infty \subset \R^n$ in the sense of Definition \ref{def:converge}.
\item[(2)] For any $a<1/2$, $b>2$, $c<5$, $d<10$, 
$T_\infty$ satisfies 
$P_1(a)$, $P_2(b)$, $P_3(c)$, $P_4(d)$.
\end{enumerate}
\end{lem}
\begin{proof}
Fix aribitrary $a<1/2$.
Then, for any $r>0$, $T_i \cap B^n(r)$ satisfies $P_1(a)$ for sufficiently large $i$.
Hence $\sharp \bigl(T_i \cap B^n(r)\bigr)$ is bounded uniformly on $i$. 
Therefore, up to subsequence $(T_i \cap B^n(r))_i$ converges to a certain finite subset in the closure of $B^n(r)$. 
Then, the diagonal argument proves (1).
(2) is an immediate consequence of the assumption that 
$S_i$ satisfies $P'(\varepsilon_i, \delta, \theta)$ for each $i$.
\end{proof}

Finally, we prove Theorem \ref{thm:Delaunay2}: 

\begin{proof}[\textbf{Proof of Theorem \ref{thm:Delaunay2}:}]
First we show that $(X_S, F_S)$ is a triangulation when $\varepsilon$ is sufficiently small. 
By Remark \ref{rem:munkres}, it is enough to show that 
$F_S$ is a nondegenerate homeomorphism.
Moreover, since $F_S$ is continuous and $|X_S|$ is compact, 
it is enough to show the following two assertions:
\begin{enumerate}
\item[(1)]  When $\varepsilon>0$ is sufficiently small, $F_S$ is nondegenerate.
\item[(2)]  When $\varepsilon>0$ is sufficiently small, $F_S$ is a bijection.
\end{enumerate}
(1) follows from Theorem \ref{thm:Delaunay} and Lemma \ref{lem:scaling}
(notice that one can take $b,c,d$ so that $2<b<c<5$, $2b<d<10$).

We prove (2). For any $s \in S$, we define $X_s \subset X_S$ by 
\[
V(X_s):= S \cap B_N(s:100\varepsilon), \qquad
\Sigma(X_s):= \Sigma(X_S) \cap 2^{V(X_s)}.
\]
Notice that the following two assertions follow from Theorem \ref{thm:Delaunay} and Lemma \ref{lem:scaling}
(to prove (2'), we also need the fact that a $C^1$-approximation of an embedding (i.e. an injective immersion) is also an embedding: 
see Theorem 8.8 in \cite{Munkres}). 
\begin{itemize}
\item[(2'):] When $\varepsilon>0$ is sufficiently small,  $F_S|_{|X_s|}$ is injective for any $s \in S$.
\item[(2''):] When $\varepsilon>0$ is sufficiently small, $B_N(s:2\varepsilon) \subset F_S(|X_s|)$ for any $s \in S$.
\end{itemize}
Now we show that (2) follows from (2') and (2'').
Suppose that $\varepsilon>0$ is so small that 
$F_S|_{|X_s|}$ is injective, and $B_N(s: 2\varepsilon) \subset F_S(|X_s|)$ for any $s \in S$.
If $F_S$ is not injective, there exist $x,y \in |X_S|$ such that $x \ne y$ and $F_S(x)=F_S(y)$.
Take $\sigma, \tau \in \Sigma(X_S)$ such that 
$x \in |\sigma|, y \in |\tau|$. 
Since $F_S(|\sigma|) \cap F_S(|\tau|) \ne \emptyset$, 
Lemma \ref{lem:later} (2) implies that 
$\tau \subset B_N(s: 100\varepsilon)$ for any $s \in \sigma$.
However it is a contradiction, since $F_S|_{|X_s|}$ is injective for any $s \in S$.
On the other hand, since $S$ satisfies $P'_2(2\varepsilon)$, 
$\bigcup_{s \in S} B_N(s:2\varepsilon) = N$.
Hence $F_S$ is surjective.
We have proved that $(X_S, F_S)$ is a triangulation for sufficiently small $\varepsilon$. 

Finally we show that there exist $c_0, c_1, c_2, c_3$ such that 
$(X_S, F_S)$ is a good triangulation with respect to $\varepsilon, c_0, c_1, c_2, c_3$
(i.e. it satisfies (1)-(4) in Definition \ref{defn:triangulation}), 
when $\varepsilon$ is sufficiently small. 
(1) was confirmed in Lemma \ref{lem:later} (1).
(2), (3), (4) are immediate consequences of our assumption that $S$ satisfies $P'(\varepsilon, \delta, \theta)$. 
\end{proof}

\section{Proof of Proposition A} 

In the following of this paper, we fix constants $c_0, c_1, c_2, c_3$ which appeared in Lemma \ref{lem:triangulation}. 
We abbreviate "good triangulation with respect to $\varepsilon, c_0, c_1, c_2, c_3$" as 
"good triangulation with respect to $\varepsilon$". 

First let us rephrase Proposition A, 
using terms introduced in Section 3. 

\begin{prop}\label{prop:gizagiza}
Let $N$ be a closed Riemannian manifold, 
$(X,F)$ be a good triangulation of $N$ with respect to $\varepsilon>0$. Then, there holds 
\[
w\bigl(N \setminus F(V(X))\bigr) \le \const_n \varepsilon.
\]
\end{prop}

By (3) in Definition \ref{defn:triangulation}, there exists a map
$h:V(X) \to \{ j \in \Z \mid 0 \le j \le c_2 \}$ 
such that 
for any simplex $\sigma=\{v_0, \ldots, v_k\}$ of $X$, 
$h(v_0), \ldots, h(v_k)$ are distinct.

Then we extend $h$ to a continuous function on $|X|$ (still denoted by $h$) as follows:
for any simplex $\sigma=\{v_0, \ldots,v_k\}$ of $X$, $h|_{|\sigma|}$ is defined by 
\[
h\Biggl(\sum_{0 \le j \le k} t_jv_j \Biggr) 
:= \sum_{0 \le j \le k} t_j h(v_j).
\]
We define a continuous function $h'$ on $N$ by $h':=h \circ F^{-1}$.
Obviously $0 \le h' \le c_2$.
Although $h'$ is not of $C^\infty$, 
it is of $C^\infty$ on $U:= \bigcup_{\sigma \in \Sigma_n(X)} F(\interior |\sigma|)$.

Let $\rho$ be a $\R_{\ge 0}$ valued smooth function defined on $\R_{\ge 0}$, such that 
$\rho$ is constant near $0$, $\supp \rho \subset [0,1]$ and 
\[
\int_{\R^n} \rho\bigl(|x|\bigr) dx =1.
\]
For $\delta>0$, let us define $\rho_\delta \in C_0^\infty(\R_{\ge 0})$ by 
$\rho_\delta(t):= \delta^{-n}\rho(t/\delta)$.
When $0<\delta<\inj(N)$, the following formula defines 
a $C^\infty$ function $h_\delta$ on $N$:
\[
h_\delta(x):= \int_{T_xN} h'\bigl(\exp_x(\zeta)\bigr) \rho_\delta\bigl(|\zeta|\bigr)\,d\vol_x(\zeta).
\]
$\vol_x$ denotes the volume form on $T_xN$ defined by the Riemannian metric on $N$.
$0 \le h' \le c_2$ implies that $0 \le h_\delta \le c_2$. 
We prove the following lemma:

\begin{lem}\label{lem:hdelta}
For any compact set $K \subset N \setminus F(V(X))$, 
$\liminf_{\delta \to 0} \min_{x \in K} |dh_\delta(x)| \ge \const_n \varepsilon^{-1}$.
\end{lem}

First we show that Proposition \ref{prop:gizagiza} follows from Lemma \ref{lem:hdelta}.
Denote the constant in Lemma \ref{lem:hdelta} by $c$. 
Then, for any $c' \in (0,c)$, 
$\min_{x \in K} |dh_\delta(x)| \ge c'\varepsilon^{-1}$ for sufficiently small $\delta$. 
Setting $h:=c'^{-1} \varepsilon h_\delta$, $h$ satisfies 
$\| h\| \le c'^{-1}c_2 \varepsilon$ and $\min_{x\in K} |dh(x)| \ge 1$. 
Hence 
\[
w\bigl(N \setminus F(V(X)) \bigr) = \sup_K w_{N \setminus F(V(X))}(K) \le \const_n \varepsilon.
\]
To prove Lemma \ref{lem:hdelta}, first we need the following sublemma.

\begin{lem}\label{lem:vx}
For any $x \in N \setminus F(V(X))$,  
there exists $\xi \in T_xN$ such that: 
$h'$ is differentiable at $x$ in the direction $\xi$, and 
$dh'(\xi)/|\xi| \ge (\sqrt{2}c_1 \varepsilon)^{-1}$. 
\end{lem}
\begin{proof}
Let $\sigma=\{v_0,\ldots,v_k\}$ be the unique simplex of $X$ such that 
$F^{-1}(x) \in \interior|\sigma|$.
Since $x \notin F(V(X))$, $k \ge 1$.
Since $h(v_0), \ldots, h(v_k)$ are distinct integers and $\diam |\sigma| \le \sqrt{2}$, 
there exists $\eta \in T_{F^{-1}(x)}|\sigma|$ such that $dh(\eta)/|\eta| \ge 1/\sqrt{2}$.
Setting $\xi:=dF(\eta)$,  
(2) in Definition \ref{defn:triangulation} implies that 
$dh'(\xi)/|\xi| \ge (\sqrt{2}c_1\varepsilon)^{-1}$.
\end{proof}

\begin{proof}[\textbf{\textit{Proof of Lemma \ref{lem:hdelta}}}]
Let us define a map $e$ by 
\[
e: \bigl\{(x,\xi) \in TN \bigm{|} |\xi| \le \inj(N) \bigr\} \to N; \qquad
(x,\xi) \mapsto \exp_x(\xi).
\]
For any $\xi \in TN$, let $\tilde{\xi}$ be its horizontal lift (with respect to the Levi-Civita connection). Then,
\[
dh_\delta(\xi) = \int_{T_xN} dh'\bigl(de(\tilde{\xi}(\zeta)) \bigr) \rho_\delta\bigl(|\zeta|\bigr)\,d\vol_x(\zeta).
\]
Since $h'$ is smooth on $U=\bigcup_{\sigma \in \Sigma_n(X)} F(\interior |\sigma|)$ (hence almost everywhere on $N$),  
the right hand side makes sense.
For each $x \in K$, Lemma \ref{lem:vx} shows that there exists a vector field $\xi$ defined near $x$, such that 
$h'$ is differentiable at $x$ in the direction $\xi(x)$, and $dh'(\xi(x))/|\xi(x)| \ge (\sqrt{2}c_1 \varepsilon)^{-1}$.
Fix $c$ so that $0<c<(\sqrt{2}c_1)^{-1}$. 
Since
$\lim_{(y,\zeta) \to (x,0)} de(\tilde{\xi}(y,\zeta)) = de(\tilde{\xi}(x,0))= \xi(x)$, 
the following inequality holds for 
sufficiently small $r, \delta>0$:
\[
y \in B_N(x:r), \, \zeta \in T_yN, \, |\zeta| \le \delta, e(y,\zeta) \in U \implies 
dh'\bigl(de(\tilde{\xi}(\zeta))\bigr)/|de(\tilde{\xi}(\zeta))| \ge c \varepsilon^{-1}.
\]
Moreover, by taking $r, \delta$ sufficiently small, we may also assume that the following holds:
\begin{align*}
\big\lvert de(\tilde{\xi}(\zeta)) \big\rvert/\big\lvert \xi(x)\big\rvert &\ge 1/2 \qquad \bigl(
y \in B_N(x:r), \zeta \in T_yN, |\zeta| \le \delta \bigr),\\
\big\lvert \xi(z) \big\rvert/ \big\lvert \xi (x) \big\rvert &\le 2  \qquad\quad \bigl(
z \in B_N(x:r) \bigr).
\end{align*}
We denote $\xi, r, \delta$ by $\xi_x, r_x, \delta_x$.

Since $K$ is compact, there exist finitely many points $x_1, \ldots, x_m \in K$ such that
$\bigl\{ B_N(x_i:r_{x_i}) \bigr\}_{1 \le i \le m}$ covers $K$.
Let $\delta:= \min_{1 \le i \le m} \delta_{x_i}$.
For any $y \in K$, there exists $1 \le i \le m$ such that 
$y \in B_N(x_i:r_{x_i})$. Then, we get 
\begin{align*}
dh_\delta\bigl(\xi_{x_i}(y)\bigr)&=\int_{T_yN} dh'\bigl(de(\tilde{\xi_{x_i}}(\zeta))\bigr) \rho_\delta(|\zeta|)\,d\vol_y(\zeta) \\
&\ge c\varepsilon^{-1} \int_{T_yN} \big\lvert de(\tilde{\xi}_{x_i}(\zeta))\big\rvert \rho_\delta(|\zeta|)\,d\vol_y(\zeta) \\
&\ge c(2\varepsilon)^{-1} \big\lvert\xi_{x_i}(x_i)\big\rvert \ge c(4\varepsilon)^{-1}\big\lvert \xi_{x_i}(y) \big\rvert.
\end{align*}
Hence $|dh_\delta(y)| \ge c(4\varepsilon)^{-1}$ for any $y \in K$.
\end{proof}

\section{Proof of Proposition B} 

First we rephrase Proposition B, using terms introduced in Section 3. 

\begin{prop}\label{prop:hakidasu}
For any $n$-dimensional connected closed Riemannian manifold $N$, there exists $\varepsilon(N)>0$ which satisfies the following property:
\begin{quote}
Let $(X,F)$ be a good triangulation of $N$ with respect to $\varepsilon \in (0, \varepsilon(N))$. 
Then, for any $v \in V(X)$, there exists an open set $W \subset N \setminus F(V(X))$ and a diffeomorphism $\Phi: W \to N \setminus \{F(v)\}$ such that 
$d^+(\Phi) \le \const_n \diam(N) \varepsilon^{-1}$. 
\end{quote}
\end{prop}

First we show that Theorem \ref{thm:closed} follows from Proposition \ref{prop:hakidasu}.
Let $N$ be a closed, connected $n$-dimensional Riemannian manifold, and $K \subsetneq N$ be a compact set on $N$. 
We have to show that $w_N(K) \le \const_n \diam (N)$.

Let $(X,F)$ be a good triangulation of $N$ with respect to $\varepsilon>0$.  
When $\varepsilon$ is sufficiently small, $F(V(X)) \setminus K \ne \emptyset$. 

Take an arbitrary $v \in V(X) \setminus F^{-1}(K)$. When $\varepsilon$ is sufficiently small, 
there exist $W, \Phi$ as in the claim in Proposition \ref{prop:hakidasu}.
Then 
\begin{align*}
w\bigl(N \setminus \{F(v)\}\bigr) &\le \const_n \diam(N)\varepsilon^{-1} w(W) \\
&\le \const_n \diam(N)\varepsilon^{-1} w\bigl(N \setminus F(V(X)) \bigr) \\
&\le \const_n \diam(N).
\end{align*}
The first inequality follows from Lemma \ref{lem:gg'}, and 
the last inequality follows from Proposition \ref{prop:gizagiza} (that is, Proposition A). 
Since $w_N(K) = w_{N \setminus \{F(v)\}}(K) \le w\bigl(N \setminus \{F(v)\}\bigr)$, 
it completes the proof of Theorem \ref{thm:closed}.

\subsection{Sketch of the proof of Proposition \ref{prop:hakidasu}}
We sketch the proof of Proposition \ref{prop:hakidasu}.
Details are carried out in Section 6.2 - 6.4.
Let $N$ be a connected, closed Riemannian manifold. 
We denote the Riemannian metric on $N$ by $g_N$. 
As a first step, we show the following lemma in Section 6.2.

\begin{lem}\label{lem:tree}
Let $(X,F)$ be a good triangulation of $N$ with respect to $\varepsilon>0$. 
Then there exists a tree $T \subset X$ such that 
$V(T)=V(X)$ and 
\[
\diam \bigl(|T|, F^*g_N|_{|T|}\bigr) \le \const_n \diam (N).
\]
\end{lem}
\begin{rem}
\textit{Tree} means a simply-connected $1$-dimensional simplicial complex.
Moreover, in Lemma \ref{lem:tree}, distance of two points $x, y \in |T|$ is defined to be the infimum of lengths of 
piecewise linear paths connecting $x$ and $y$, where lengths of path on $|T|$ is defined by the metric $F^*g_N|_{|T|}$. 
\end{rem}

Let $(X,F)$ be a good triangulation of $N$ with respect to $\varepsilon>0$, and 
take a tree $T \subset X$ as in Lemma \ref{lem:tree}. 
Our next step is to 
define a compact Riemannian manifold (with boundary) $\tilde{T}$ and an embedding map $i:|T| \to \tilde{T}$.
Roughly speaking, $\tilde{T}$ is obtained by "fatten" an embedded tree $F(|T|) \subset N$, i.e. 
we replace vertices with closed balls, and 
replace edges with cylinders. 
A rigorous definition is described in Section 6.3. 
Then we show the following two lemmas
(Lemma \ref{lem:umekomi} is proved in Section 6.3, and Lemma \ref{lem:henkei} is proved in Section 6.4.):

\begin{lem}\label{lem:umekomi}
For any $0<\delta<1$, there exists $\varepsilon(\delta,N)>0$ such that the following holds:
If $0< \varepsilon < \varepsilon(\delta,N)$, 
there exists an embedding $I: \tilde{T} \to N$ such that:
\begin{enumerate}
\item[(1)] $I \circ i :|T| \to N$ is equal to $F|_{|T|}$.
\item[(2)] $1-\delta \le d^-(I)$ and $d^+(I) \le 1+\delta$. 
\end{enumerate}
\end{lem}

\begin{lem}\label{lem:henkei}
For any $v \in V(T)$, there exists a neighborhood $Z$ of $\partial\tilde{T}$ in $\tilde{T}$, and a
diffeomorphism $\varphi:Z \to \tilde{T} \setminus \bigl\{ i(v) \bigr\}$ with the following properties: 
\begin{enumerate}
\item[(1)] $\varphi \equiv \id$ on some neighborhood of $\partial \tilde{T}$.
\item[(2)] $d^+(\varphi) \le \const_n\diam(|T|)\varepsilon^{-1}$, where a metric on $Z$ is a restriction of the metric on $\tilde{T}$. 
\end{enumerate}
\end{lem}

We prove Proposition \ref{prop:hakidasu} assuming those results.
We claim that 
\[
\varepsilon(N):= \min \{ \varepsilon(1/2, N), \diam(N) \}
\]
satisfies the requirement in Proposition \ref{prop:hakidasu}. 

Suppose that $0 < \varepsilon < \varepsilon(N)$. 
Take 
$I:\tilde{T} \to N$ as in Lemma \ref{lem:umekomi}, and 
$Z$, $\varphi$ as in Lemma \ref{lem:henkei}. 
Then define $W \subset N$ by 
$W:=\bigl( N \setminus I(\tilde{T}) \bigr) \cup I(Z)$, and 
define $\Phi: W \to N \setminus \{F(v)\}$ by 
\[
\Phi(x) = \begin{cases} 
          x &\bigl(x  \in N \setminus I(\tilde{T}) \bigr) \\
          I \circ \varphi \circ I^{-1}(x) &\bigl(x \in I(Z)\bigr) 
          \end{cases}.
\]

We have to check that $d^+(\Phi) \le \const_n \diam(N) \varepsilon^{-1}$, i.e. 
any $x \in W$ and $\xi \in T_xW$ satisfy
$|d\Phi(\xi)| \le \const_n\diam(N)\varepsilon^{-1} |\xi|$.
If $x \notin I(\tilde{T})$, using $\varepsilon(N) \le \diam(N)$ we obtain 
\[
|d\Phi(\xi)| = |\xi|  \le \diam(N) \varepsilon^{-1} |\xi|.
\]
On the other hand, if $x \in I(Z)$, then 
\[
|d\Phi(\xi)| \le \const_n \diam(|T|) \varepsilon^{-1} |\xi| \le 
\const_n\diam(N)\varepsilon^{-1} |\xi|.  
\]
The first inequality follows from Lemma \ref{lem:umekomi} (2) and Lemma \ref{lem:henkei} (2), 
the second inequality follows from Lemma \ref{lem:tree}.

\subsection{Proof of Lemma \ref{lem:tree}}
Since $(X,F)$ satisfies Definition \ref{defn:triangulation} (2), 
Lemma \ref{lem:tree} follows from the following lemma:

\begin{lem}\label{lem:sympcpx}
Let $X$ be a connected simplicial complex. 
Then, there exists a tree $T \subset X$ such that 
$V(T)=V(X)$ and $\diam(|T|) \le \const_n \diam(|X|)$, where 
$\diam(|T|), \diam(|X|)$ are defined with respect to 
the standard metrics (see Section 3).
\end{lem}
\begin{proof}
First we show the following claim:
\begin{quote}
For any $k \in \{2, \ldots, n\}$, $\diam(|X_{k-1}|) \le \const_k \cdot \diam(|X_k|)$.
\end{quote}
Let $c:[0,1] \to |X_k|$ be a piecewise linear map such that $c(0), c(1) \in |X_{k-1}|$.
Then there exists $0=t_0<t_1< \cdots < t_m=1$ such that:
\begin{itemize}
\item $c(t_0), \ldots,c(t_m) \in |X_{k-1}|$.
\item For any $i=1,\ldots, m$ there exists $\sigma_i \in \Sigma_k(X)$ such that 
$c\bigl([t_{i-1}, t_i]\bigr) \subset |\sigma_i|$.
\end{itemize}
For each $i=1,\ldots,m$, there exists $c_i:[t_{i-1},t_i] \to |\partial \sigma_i|$ such that 
$c_i(t_{i-1})=c(t_{i-1})$, 
$c_i(t_i)=c(t_i)$ and 
$l(c_i) \le \const_k \cdot  l(c|_{[t_{i-1},t_i]})$
($l$ denotes the lengths of curves).
By connecting $c_1, \ldots, c_m$ we get a map $c':[0,1] \to |X_{k-1}|$ such that 
$c'(0)=c(0)$, $c'(1)=c(1)$, 
$l(c') \le \const_k \cdot l(c)$.
Hence we have proved the above claim.
By applying the above claim for $k=2,\ldots,n$, we get 
$\diam(|X_1|) \le \const_n \cdot \diam(|X|)$.

Take an arbitrary function $\rho: \Sigma_1(X) \to [1,2]$ such that $\bigl\{\rho(\sigma)\bigr\}_{\sigma \in \Sigma_1(X)}$ are linearly independent over $\Q$.
\textit{A path on $X$} means a subcomplex of $X$ which is isomorphic (as a simplicial complex) to some 
$P_l\,(l=1,2,\ldots)$, where $P_l$ is defined as 
\[
V(P_l)=\{0,\ldots,l\}, \qquad 
\Sigma(P_l)=\bigl\{ \{0,1\}, \{1,2\}, \ldots, \{l-1,l\} \bigr\}.
\]
For any path $\gamma$ on $X$, let $\rho(\gamma):=\sum_{\sigma \in \Sigma_1(\gamma)} \rho(\sigma)$.
If two paths $\gamma, \gamma'$ satisfy $\rho(\gamma)=\rho(\gamma')$, then $\gamma=\gamma'$.

Fix an arbitrary element $v_0 \in V(X)$.
For each $v \in V(X)$, let $\gamma_v$ be the path on $X$ connecting $v$ and $v_0$, which 
attains the minimum value of $\rho$.
Then $\rho(\gamma_v) \le \const_n \cdot \diam(|X|)$, since 
$\diam(|X_1|) \le \const_n \cdot \diam(|X|)$ and 
$\rho(\sigma) \le 2 $ for any $\sigma \in \Sigma_1(X)$.

Let $T$ be the union of $\gamma_v$, where $v$ runs over all elements of $V(X)$.
Then, it is easy to check that $T$ is a tree.
Moreover, $\diam(|T|) \le \const_n \cdot \diam(|X|)$, since for any $v,w \in V(X)$
\[
\dist_{|T|}(v,w) \le \dist_{|T|}(v,v_0) + \dist_{|T|}(v_0,w) \le \const_n \cdot  \diam(|X|).
\]
\end{proof}

\subsection{Definition of $\tilde{T}$ and $i$}

Let $N$ be a closed Riemannian manifold, and $(X,F)$ be a good triangulation of $N$ with respect to $\varepsilon>0$. 
Let $T \subset X$ be a tree as in Lemma \ref{lem:tree}. 
Our aim in this section is to define a compact Riemannian manifold (with boundary) $\tilde{T}$ and an embedding map $i: |T| \to \tilde{T}$. 
We also prove Lemma \ref{lem:umekomi}. 

We fix a continuous function $\mu:[0,\infty) \to [0,\infty)$ satisfying the following properties:
\begin{itemize}
\item $\mu(t)=c_3t$ for $0 \le t \le 1/\sqrt{1+{c_3}^2}$.
\item $\mu(t)=\sqrt{1-t^2}$ when $t \ge 1/\sqrt{1+{c_3}^2}$ and $t$ is sufficiently close to $1/\sqrt{1+{c_3}^2}$.
\item $\mu(t) \equiv c_3/2\sqrt{1+{c_3}^2}$ when $t \ge 1$.
\item $\mu(t)$ is a non-increasing function for $t \ge 1/\sqrt{1+{c_3}^2}$.
\end{itemize}

In the following, we abbreviate $F(u) \in N$ as $u$ for any $u \in V(T)$. 
Let $u \in V(T)$ , $v \in N_T(u)$ and $r>0$. 
We define $A^-_{uv}(r), A^+_{uv}(r), A_{uv}(r) \subset T_uN$ as follows
($d_{uv}$ abbreviates $\dist_N(u,v)$. $H_{uv} \subset T_uN$ denotes the orthogonal complement of $\overrightarrow{uv} \cdot \R$): 
\begin{align*}
A_{uv}(r)&:= \bigl\{ h+te_{uv} \bigm{|} 0 \le t \le 2d_{uv}/3, \, h \in H_{uv},\, |h| \le r \mu(t/r) \bigr\}, \\
A^-_{uv}(r)&:= \bigl\{ h+ te_{uv} \bigm{|} 0 \le t \le d_{uv}/3, \, h \in H_{uv}, \, |h| \le r \mu(t/r) \bigr\}, \\
A^+_{uv}(r)&:= \bigl\{ h+ te_{uv} \bigm{|} d_{uv}/3 \le t \le 2d_{uv}/3, h \in H_{uv}, \, |h| \le r \mu(t/r) \bigr\}.
\end{align*}
We equip $A_{uv}(r), A^{\pm}_{uv}(r)$ with the metric on $T_uN$.
\begin{center}
\input{AUV.tpc}

Fig.1:
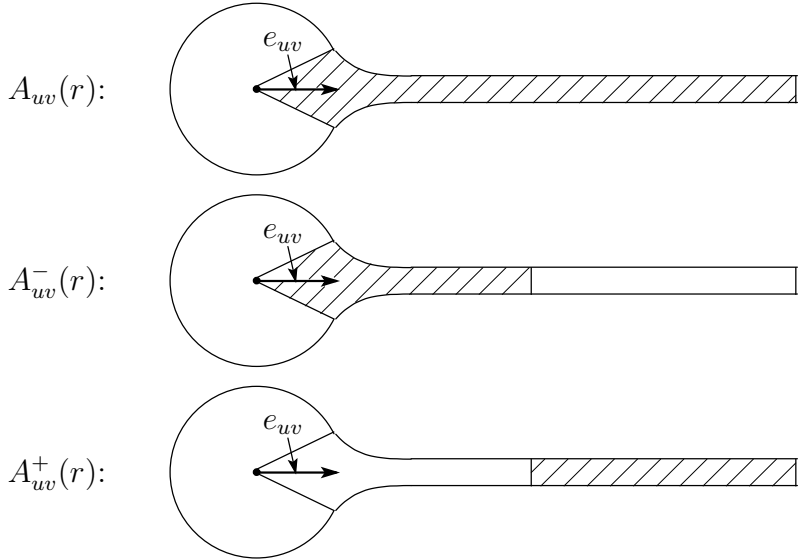
shaded regions are $A_{uv}(r)$, $A_{uv}^-(r)$, $A_{uv}^+(r)$. 
\end{center}

Since the triangulation $(X,F)$ satisfies Definition \ref{defn:triangulation} (4), there holds
\[
u \in V(T),\quad v, w \in N_T(u),\quad v \ne w \implies A_{uv}(r) \cap A_{uw}(r) = \{0\}, 
\]
where $0$ denotes the origin of $T_uN$. 

Let $i: H_{uv} \to H_{vu}$ be an isometry defined by a parallel transport along the shortest geodesic segment 
connecting $u$ and $v$.
When $r<d_{uv}/3$, then $\mu(t/r) \equiv c_3/2\sqrt{1+c_3^2}$ for $t \ge d_{uv}/3$. 
Therefore we can define an isometry $\psi_{uv}: A^+_{uv}(r) \to A^+_{vu}(r)$ by 
\[
\psi_{uv}(h+te_{uv}) = i(h) + (d_{uv}-t)e_{vu}.
\]

For any $0 < r < c_0\varepsilon/3$, 
there holds $r  < d_{uv}/3$ for any $u \in V(T)$, $v \in N_T(u)$ 
since Definition \ref{defn:triangulation} (1), (2) shows $d_{uv} \ge c_0 \varepsilon$. 
For each $u \in V(T)$, let us define 
\begin{align*}
B_u(r)&:=\bigl\{x \in T_uN  \bigm{|} |x| \le r \bigr\}, \\
C_u(r)&:= B_u(r) \cup \bigcup_{v \in N_T(u)} A_{uv}(r).
\end{align*}
\begin{center}
\input{CU.tpc}

Fig.2: shaded region is $C_u(r)$. 
\end{center}

We define an equivalence relation $\sim$ on $\bigsqcup_{u \in V(T)} C_u(r)$ so that: 
$x \sim y$ if and only if $x=y$ or 
$x \in A_{uv}^+(r)$, $y \in A_{vu}^+(r)$, $y=\psi_{uv}(x)$ for some $u \in V(T), v \in N_T(u)$.

Then  we define $T(r)$, a compact Riemannian manifold with boundary, 
by $T(r):=\bigsqcup_{u \in V(T)}C_u(r)/\sim$.
Since $\psi_{uv}$ are isometories, $T(r)$ carries a natural Riemannian metric $g_{T(r)}$. 
$|\,\cdot\,|_{g_{T(r)}}$ is abbreviated as $|\,\cdot\,|_{T(r)}$.
We define $i_r:|T| \to T(r)$ by 
\[
i_r\bigl((1-t)u + tv \bigr):= \bigl[ t \overrightarrow{uv} \bigr] \qquad (0 \le t \le 1),
\]
where $u \in V(T)$, $v \in N_T(u)$ (since $\psi_{uv}\bigl(t \overrightarrow{uv}\bigr)=(1-t)\overrightarrow{vu}$, 
this is well-defined).

\begin{rem}\label{rem:subset}
The following remarks are immediate consequences of the definition:
\begin{itemize}
\item If $r' \le r$, there exists a natural isometric embedding $T(r') \to T(r)$.
In the following, we regard $T(r')$ as a submanifold of $T(r)$. 
\item If $S$ is a subtree of $T$, there exists a natural isometric embedding $S(r) \to T(r)$.
In the following, we regard $S(r)$ as a submanifold of $T(r)$. 
\end{itemize}
\end{rem}

To define $\tilde{T}$ and $i:|T| \to \tilde{T}$, we need the following lemma: 

\begin{lem}\label{lem:umekomi2}
There exists a positive constant $\rho_n$ depending only on $n$, 
which satisfies the following properties: 
\begin{quote}
For any $0<\delta<1$, there exists $\varepsilon(\delta,N)>0$ such that:
if $0<\varepsilon<\varepsilon(\delta,N)$ and $r \le \varepsilon\rho_n$, 
there exists an embedding $I: T(r) \to N$ satisfying 
\begin{enumerate}
\item[(a):] $I \circ i_r:|T| \to N$ is equal to $F|_{|T|}$.
\item[(b):] $1-\delta \le d^-(I)$ and $d^+(I) \le 1+\delta$. 
\end{enumerate}
\end{quote}
\end{lem}

We fix $\rho_n>0$ as in Lemma \ref{lem:umekomi2}, and define $\tilde{T}$ and $i:|T| \to \tilde{T}$ as 
\[
\tilde{T}:=T(\varepsilon \rho_n), \qquad i:=i_{\varepsilon \rho_n}.
\]
Then, it is clear that Lemma \ref{lem:umekomi} holds. 

\begin{proof}[\textbf{\textit{Proof of Lemma \ref{lem:umekomi2}}}]
Setting cut off function $\chi: [1/3, 2/3] \to [0,1]$ such that 
$\chi \equiv 0$ near $1/3$, $\chi \equiv 1$ near $2/3$ and 
$\chi(t) + \chi(1-t) =1$,
we define $I:T(r) \to N$ as follows: 
\begin{enumerate}
\item[(i)] For any $u \in V(T)$, $I\bigl([x]\bigr):=\exp_u(x)$ where $x \in C_u^-(r):=B_u(r) \cup \bigcup_{v \in N_T(u)}A^-_{uv}(r)$.
\item[(ii)] For any $u \in V(T)$, $v \in N_T(u)$ and 
$x=h+t\overrightarrow{uv} \in A^+_{uv}(r)$, 
\[
I([x]):= \gamma_{\exp_u(x)\exp_v(\psi_{uv}(x))}\bigl(\chi(t)\bigr).
\]
Recall that $\gamma$ denotes the shortest geodesic (see Section 1.3). 
Since $\chi(t)+\chi(1-t)=1$, this is well-defined.
\end{enumerate}
It is clear from the above construction that for any $u \in V(T)$ and $v \in N_T(u)$, 
$I(t \overrightarrow{uv}) = \gamma_{uv}(t)$. 
On the other hand, since $(X_S, F_S)$ satisfies Definition \ref{defn:triangulation} (1), 
$F((1-t)u+tv)= \gamma_{uv}(t)$. Hence (a) holds. 

If the metric of $N$ is flat, $I$ is isometric. 
Therefore $I$ satisfies (b) with $\delta>0$ when $\varepsilon$ is sufficiently small. 
In particular, $I$ is an immersion for sufficiently small $\varepsilon$. 

Finally, we have to specify $\rho_n$ so that if $r/\varepsilon \le \rho_n$, then $I$ is injective for sufficiently small $\varepsilon$. 
For each $u \in V(T)$, define a tree $T_u \subset T$ by 
\[
V(T_u):= \{u\} \cup N_T(u),  \qquad
\Sigma_1(T_u):= \bigl\{ \{u,v\} \bigm{|} v \in N_T(u) \bigr\}.
\]
Following Remark \ref{rem:subset}, 
we consider $T_u(r)$ as a submanifold of $T(r)$.
It is easy to see that if $r/\varepsilon < c_0c_3/\sqrt{1+c_3^2}$ and $\varepsilon$ is sufficiently small, then 
$I|_{T_u(r)}$ is injective. 

On the other hand, for each $e \in \Sigma_1(T)$, 
$\bar{e}$ denotes the subtree of $T$ consisting of $e$ and two vertices of $e$. 
Following Remark \ref{rem:subset}, 
we consider $\bar{e}(r)$ as a submanifold of $T(r)$. 
It is easy to see that if $r/\varepsilon < c_0/2$ and $\varepsilon$ is sufficiently small, 
the following holds: 
\[
e, e' \in \Sigma_1(T), e \cap e' = \emptyset \implies I(\bar{e}(r)) \cap I(\bar{e'}(r)) = \emptyset. 
\]

We show that $I$ is injective when $r/\varepsilon < \min \bigl\{ c_0c_3/\sqrt{1+c_3^2}, c_0/2  \bigr\} $ and $\varepsilon$ is sufficiently small. 
If $I$ is not injective, there exists $x,y \in T(r)$ such that $x \ne y$ and $I(x)=I(y)$. 
Take $e, e' \in \Sigma_1(T)$ such that $x \in \bar{e}(r)$, $y \in \bar{e'}(r)$. 
Then obviously $I(\bar{e}(r)) \cap I(\bar{e'}(r)) \ne \emptyset$, hence $e \cap e' \ne \emptyset$. 
Taking $u \in e \cap e'$, $x,y \in T_u(r)$.
This is a contradiction, since $I|_{T_u(r)}$ is injective.
\end{proof}

\subsection{Proof of Lemma \ref{lem:henkei}}
Let us denote $r:=\rho_n\varepsilon$, i.e. $\tilde{T}=T(r)$.
First we recall what we have to show (we denote $v$ in the statement of Lemma \ref{lem:henkei} as $v_1$):
\begin{quote}
For any $v_1 \in V(T)$, 
there exists $Z \subset T(r)$, a neighborhood of $\partial T(r)$, and a 
diffeomorphism $\varphi: Z \to T(r) \setminus \bigl\{i(v_1)\bigr\}$, such that
$d^+(\varphi) \le \const_n \diam(|T|) \varepsilon^{-1}$ and $\varphi$ is an identity on some neighborhood of $\partial T(r)$. 
\end{quote}
We use abbreviations 
$d:=\diam(|T|)$ and 
$d(v):=\dist_{|T|}(v,v_1)\,(\forall v \in V(T))$.

\begin{rem}\label{rem:lengths}
For any $e \in \Sigma_1(T)$,  
$(\text{length of $F(|e|)$})/\varepsilon$ is uniformly bounded, since $F$ satisfies 
(2) in Definition \ref{defn:triangulation}.
Therefore, it is enough to prove Lemma \ref{lem:henkei} assuming that 
lengths of $F(|e|)$ are same for all $e \in \Sigma_1(T)$.
\end{rem}

For any subtree $S$ of $T$ such that $v_1 \in V(S)$, 
let $\nu_S$ be the normal vector of $\partial S(r/2)$ which points inward with respect to $S(r/2)$. 
We regard $S(r/2)$ as a submanifold of $S(r)$ (see Remark \ref{rem:subset}). 
For sufficiently small $c>0$, 
there exists an embedding 
$E_S: \partial S(r/2) \times (-cr, cr) \to S(r)$ such that ($t$ denotes the coordinate on $(-cr,cr)$):
\[
E_S(z,0)=z, \qquad \partial_t E_S(z,0) = \nu_S(z), \qquad \partial_t^2 E_S(z,t)=0.
\]
In the last equation, $\partial_t^2$ is defined by the Levi-Civita connection associated with $g_{S(r)}$.
Since $\nu_S$ points inward, $E_S^{-1}(S(r/2))=\partial S(r/2) \times [0,cr)$. 
Note that we may take $c>0$ so that it depends only on $n$ (hence, independent on $T$ and $S$).
We fix such $c$ and denote it as $c_4$.

Define a manifold $X_S$ by 
\[
X_S:= \bigl(T(r) \setminus S(r/2)\bigr) \bigcup_{E_S|_{\partial S(r/2) \times (-c_4r,0]}} \partial S(r/2) \times (-c_4r, d).
\]
We equip $X_S$ with a metric $g_{X_S}$, which is defined in the following manner.
First, we define a metric $g$ on $\partial S(r/2) \times (-c_4r, d\,)$ as follows 
($\pr_{\partial S(r/2)}$ denotes the projection to $\partial S(r/2)$).
\[
g:= (\pr_{\partial S(r/2)})^* (g_{T(r)}|_{\partial S(r/2)}) + dt^2.
\]
Setting cut off function $\alpha: (-c_4,0] \to [0,1]$ such that
$\alpha \equiv 1$ near $-c_4$ and $\alpha \equiv 0$ near $0$, 
we define a metric $g_{X_S}$ on $X_S$ so that 
\begin{itemize}
\item $g_{X_S}=g_{T(r)}$ on $\bigl(T(r) \setminus S(r/2)\bigr) \setminus \Ima E_S$.
\item $g_{X_S}= \alpha(t/r) {E_S}^*(g_{T(r)}) + \bigl(1-\alpha(t/r)\bigr) g$ on 
$\partial S(r/2) \times (-c_4r,0]$.
\item $g_{X_S}= g$ on $\partial S(r/2) \times [0,d)$.
\end{itemize}

Consider the case $S=T$. Then, there exists a diffeomorphism 
\[
\kappa: \bigl(T(r) \setminus T(r/2)\bigr) \cup \Ima E_T \to X_T, 
\]
which is an identity on some neighborhood of $\partial T(r)$, and satisfies 
$d^+(\kappa: g_{T(r)}, g_{X_T}) \le \const_n d \varepsilon^{-1}$.
Hence it is enough to show the following lemma:

\begin{lem}\label{lem:W}
There exists $Y \subset X_T$, a neighborhood of $\partial T(r)$ in $X_T$, 
and a diffeomorphism $\varphi':Y \to T(r) \setminus \{i(v_1)\}$ such that $\varphi'$ is an identity on some neighborhood of $\partial T(r)$, and 
$d^+(\varphi': g_{X_T}, g_{T(r)}) \le \const_n$.
\end{lem}

Actually, once we prove Lemma \ref{lem:W}, $Z:=\kappa^{-1}(Y)$ and $\varphi:=\varphi' \circ \kappa$ satisfy 
the requirements of Lemma \ref{lem:henkei}. 

To prove Lemma \ref{lem:W}, first we define $Y \subset X_T$.
Fix a cut off function $\chi: [1/3, 2/3] \to [0,1]$ such that 
$\chi \equiv 1$ near $1/3$, $\chi \equiv 0$ near $2/3$ and 
$\chi(t) + \chi(1-t)=1$.
We define $\bar{d} \in C^\infty\bigl(T(r)\bigr)$ as follows:
\begin{itemize}
\item For each $v \in V(T)$, $\bar{d} \equiv d(v)$ on 
$C^-_v(r):=B_v(r) \cup \bigcup_{w \in N_T(v)} A^-_{vw}(r)$.
\item For each $v \in V(T)$, $w \in N_T(v)$ and 
$x=h+ t\overrightarrow{vw} \in A^+_{vw}(r)$, 
\[
\bar{d}\bigl([x]\bigr):= \chi(t) d(v) + \chi(1-t) d(w).
\]
Since $\chi(t)+\chi(1-t)=1$, this is well-defined.
\end{itemize}
Let $S$ be a subtree of $T$, satisfying $v_1 \in V(S)$. 
We define $Y_S \subset X_S$ by 
\[
Y_S:=\bigl(T(r) \setminus S(r/2)\bigr) \cup \bigl\{(z,t) \bigm{|} z \in \partial S(r/2),
0 \le t < \bar{d}(z) \bigr\}, 
\]
and define $Y \subset X_T$ by $Y:=Y_T$.

Next we construct a diffeomorphism 
$\varphi': Y_T \to T(r) \setminus \{i(v_1)\}$, 
which satisfies the requirements of Lemma \ref{lem:W}.

Let $m:= \big\lvert V(T) \big\rvert$, 
and take an arbitrary increasing sequence of subtrees of $T$: 
\[
\{v_1\}=T_1 \subset T_2 \subset \cdots \subset T_m = T.
\]
Our idea is to construct a sequence of diffeomorphisms between Riemannian manifolds
\[
Y_T = Y_{T_m} \to Y_{T_{m-1}} \to \cdots \to Y_{T_2} \to Y_{T_1} \to T(r) \setminus \{i(v_1)\}. 
\]
We abbreviate the Riemannian metric $g_{X_{T_j}}|_{Y_j}$ as $g_j$. 
To spell out the proof, we introduce some notations:
\begin{itemize}
\item
For any $u \in V(T)$, $v \in N_T(u)$ and $0 \le a \le r$, we define 
$R^0_{uv}(a), R^1_{uv}(a) \subset B_u(r)$ by 
\begin{align*}
R^0_{uv}(a)&:=\bigl\{ h + te_{uv} \bigm{|}h \in H_{uv}, \,  \sqrt{|h|^2+t^2}=a, \, |h| \le c_3 t \bigr\}, \\
R^1_{uv}(a)&:=\bigl\{ h + te_{uv} \bigm{|}h \in H_{uv}, \, \sqrt{|h|^2+t^2}=a, \, |h| \ge c_3 t \bigr\}.
\end{align*}
Moreover, for $0 \le b < c \le r$, we define 
$R^0_{uv}(b,c), R^1_{uv}(b,c) \subset B_u(r)$ by 
\begin{align*}
R^0_{uv}(b,c)&:=\bigl\{ h+ te_{uv} \bigm{|}h \in H_{uv}, \, b < \sqrt{|h|^2+t^2} < c,\, |h| \le c_3t \bigr\}, \\
R^1_{uv}(b,c)&:=\bigl\{ h+ te_{uv} \bigm{|}h \in H_{uv}, \, b < \sqrt{|h|^2+t^2} < c,\, |h| \ge c_3t \bigr\}.
\end{align*}
$R^i_{uv}[b,c), R^i_{uv}(b,c], R^i_{uv}[b,c]\,(i=0,1)$ are defined in similar way.

\begin{rem}\label{rem:C4}
It is easy to check from Definition \ref{defn:triangulation} (4) that 
the following holds for any $u \in V(T)$:
\[
v,w \in N_T(u),\quad  v \ne w \implies R^0_{uv}(a) \subset R^1_{uw}(a), \quad 
R^0_{uv}(b,c) \subset R^1_{uw}(b,c).
\]
\end{rem}

\item
For $1 \le j \le m$, let $v_j$ be the only element of $V(T_j) \setminus V(T_{j-1})$, 
and let $w_j$ be the only element of 
$V(T_{j-1}) \cap N_T(v_j)$.
\item We define $A_j \subset T(r)$ by 
\[
A_j:=B_{v_j}(r) \cup A_{v_jw_j}(r) \cup A_{w_jv_j}(r).
\]
\begin{center}
\input{AJ.tpc}

Fig. 3: shaded region is $A_j$. 
\end{center}
Moreover, we define $B_j, C_j \subset Y_{T_j}$ and $D_j, E_j \subset Y_{T_{j-1}}$ as follows: 
\begin{align*}
B_j&:= R^1_{v_jw_j}(r/2,r] \cup \bigl\{(z,t) \bigm{|} z \in R^1_{v_jw_j}(r/2), \, 0 \le t < d(v_j) \bigr\}, \\
C_j&:= \bigl(A_j \setminus T_j(r/2) \bigr) \cup \bigl\{(z,t) \bigm{|} z \in A_j \cap \partial T_j(r/2), \, 0 \le t  < \bar{d}(z) \bigr\}, \\
D_j&:= \bigl\{ (z,t) \bigm{|} z \in A_j \cap \partial T_{j-1}(r/2), \, 0 \le t < d(w_j) \bigr\},\\
E_j&:= \bigl(A_j \setminus T_{j-1}(r/2) \bigr) \cup D_j.
\end{align*}
\end{itemize}
\begin{rem}\label{rem:iota}
For each $j$, 
$Y_{T_j} \setminus C_j$ is naturally identified with $Y_{T_{j-1}} \setminus E_j$.
The identification map 
$\iota_j: Y_{T_j} \setminus C_j \to Y_{T_{j-1}} \setminus E_j$ preserves the metric:
$\iota_j^*g_{j-1} = g_j$. 
\end{rem}

\begin{lem}\label{lem:psiphi}
For each $j=2, \ldots,m$, there exists a diffeomorphism 
$\varphi_j:Y_{T_j} \to Y_{T_{j-1}}$ 
which satisfies the following properties:
\begin{enumerate}
\item[(a):] $\varphi_j|_{Y_{T_j} \setminus C_j} \equiv \iota_j$.
\item[(b):] $d^+(\varphi_j|_{B_j \cap \varphi_j^{-1}(D_j)}: g_j, g_{j-1}) \le 1$. 
\item[(c):] $d^+(\varphi_j: g_j, g_{j-1}) \le c_5$, where $c_5$ is a constant which depends only on $n$. 
\item[(d):] $\varphi_j(R^1_{v_jw_j}[r/2,r]) \subset A_j \setminus B_{w_j}(r)$.
\end{enumerate}
\end{lem}
\begin{rem}
In (d), notice that $A_j \setminus B_{w_j}(r) \subset Y_{T_{j-1}}$, since
$A_j \setminus B_{w_j}(r)$ is disjoint from $T_{j-1}(r/2)$.
\end{rem}
\begin{proof}
Notice that 
the boundary of 
$\partial T_j(r/2) \cap A_j$ and the boundary of 
$\partial T_{j-1}(r/2) \cap A_j$ are same, and described as
\[
\{ h+ te_{w_jv_j} \mid h \in H_{w_jv_j},\,  \sqrt{|h|^2+t^2}=r/2, \, |h| = c_3t \}.
\]
There exists a diffeomorphism $\psi: \partial T_j(r/2) \cap A_j \to \partial T_{j-1}(r/2) \cap A_j$ satisfying 
\begin{enumerate}
\item[($\psi$-1):]
$\psi$ is an identity on some neighborhood of boundaries. 
\item[($\psi$-2):]
$d^+(\psi)$ is bounded by some constant which depends only on $n$, where $\partial T_j(r/2)$, $\partial T_{j-1}(r/2)$ 
are equipped with the restriction of $g_{T(r)}$. 
\item[($\psi$-3):]
$d^+(\psi|_{R^1_{v_jw_j}(r/2)}) \le 1$. 
\end{enumerate} 
It is possible to achieve ($\psi$-2), since we assumed in Remark \ref{rem:lengths} that lengths of $F(|e|)$ are same for all $e \in \Sigma_1(T)$. 

Now, define $F_j \subset Y_{T_j}$ and a diffeomorphism $\bar{\psi}: F_j \to D_j$ by 
\begin{align*}
F_j&:=\bigl\{ (z,t) \bigm{|} z \in \partial T_j(r/2) \cap A_j, \, \bar{d}(z)-d(w_j) \le t < \bar{d}(z) \bigr\}, \\
\bar{\psi}(z,t)&:= \bigl( \psi_j(z), t- \bar{d}(z) + d(w_j) \bigr). 
\end{align*}
Then, ($\psi$-2) implies that $d^+(\bar{\psi}: g_j, g_{j-1})$ is bounded by some constant which depends only on $n$. 
Moreover, since $(z,t) \in B_j \cap F_j \implies z \in R^1_{v_jw_j}(r/2)$, ($\psi$-3) implies 
$d^+(\bar{\psi}|_{B_j \cap F_j}: g_j, g_{j-1}) \le 1$. 

Finally, we can extend $\bar{\psi}$ to $\varphi_j: Y_{T_j} \to Y_{T_{j-1}}$ so that (a)-(d) are satisfied. 
(a) can be achieved by ($\psi$-1). 
Since $B_j \cap \varphi_j^{-1}(D_j) = B_j \cap F_j$, 
(b) is immediate from $d^+(\bar{\psi}|_{B_j \cap F_j}: g_j, g_{j-1}) \le 1$. 
(c) can be achieved because 
$C_1 \setminus F_1, \ldots, C_m \setminus F_m$ are isometric, and 
$E_1 \setminus D_1, \ldots, E_m \setminus D_m$ are isometric. 
(d) can be achieved since
$R^1_{v_jw_j}[r/2, r] \cap F_j = \emptyset$. 
\end{proof}

In the following, (a), (b), (c), (d) mean conditions in Lemma \ref{lem:psiphi}.
Moreover, for $j=1,\ldots,m$, $|\, \cdot \, |_{g_j}$ is abbreviated as 
$|\, \cdot \,|_j$. 

\begin{lem}\label{lem:C53}
Let $\varphi_j: Y_{T_j} \to Y_{T_{j-1}}\,(j=2, \ldots, m)$ be diffeomorphisms which satisfy the conditions in Lemma \ref{lem:psiphi}.
For any $y \in Y_T = Y_{T_m}$ and $1 \le j \le m$, define $y_j \in Y_{T_j}$ by 
$y_m:=y$, $y_{j-1}:= \varphi_j(y_j)$. 
If there exists $\xi \in T_{y_j}Y_{T_j}$ such that 
$|d\varphi_j(\xi)|_{j-1} > |\xi|_j$, 
then at least one of the following holds:
\begin{enumerate}
\item[(A):] Any $k>j$ and $\zeta \in T_{y_k}Y_{T_k}$ satisfies $|d\varphi_k(\zeta)|_{k-1} = |\zeta|_k$.
\item[(B):] 
The number of $k<j$, such that there exists $\zeta \in T_{y_k}Y_{T_k}$ satisfying 
$|d\varphi_k(\zeta)|_{k-1} > |\zeta|_k$, is at most $1$.
\end{enumerate}
\end{lem}
\begin{proof} 
Assume that there exists $\xi \in T_{y_j}Y_{T_j}$ such that 
$|d\varphi_j(\xi)|_{j-1} > |\xi|_j$. 
(a) and Remark \ref{rem:iota} imply that $y_j \in C_j$, hence $y_{j-1} \in E_j$.
We consider the following two cases:
\begin{enumerate}
\item[(i):] There exists $u \in N_T(v_j) \setminus \{w_j\}$ such that 
$y_j \in R^0_{v_ju}(r/2,r] \cup \{(z,t)\mid z \in R^0_{v_ju}(r/2), 0 \le t <  d(v_j)\}$.
\item[(ii):] Otherwise.
\end{enumerate}
First we consider the case (i).
By Remark \ref{rem:C4}, $y_j \in B_j$.
Notice that $E_j$ is divided into three parts:
\[
D_j, \quad 
A_j \setminus B_{w_j}(r) , \quad R^0_{w_jv_j}(r/2,r].
\]
\begin{enumerate}
\item[(i)-(i):]
Assume that 
$y_{j-1} \in D_j$.
Then, $y_j \in B_j \cap \varphi_j^{-1}(D_j)$. 
Hence (b) implies $|d\varphi_j(\xi)|_{j-1} \le |\xi|_j$ for any $\xi \in T_{y_j}Y_{T_j}$, 
contradicting the assuption.
\item[(i)-(ii):]
Assume that $y_{j-1} \in A_j \setminus B_{w_j}(r)$. 
In this case, $y_k \notin C_k$ for any $k<j$.
Hence (a) and Remark \ref{rem:iota} imply that
for any $k<j$, $y_k=y_{j-1}$ and 
$|d\varphi_k(\zeta)|_{k-1}=|\zeta|_k\, (\forall \zeta \in T_{y_k}Y_{T_k})$.
Therefore (B) holds.
\item[(i)-(iii):]
Assume that $y_{j-1} \in R^0_{w_jv_j}(r/2,r]$. 
Let $k_0$ be the unique integer such that $w_j=v_{k_0}$. 
We claim that $|d\varphi_k(\zeta)|_{k-1}=|\zeta|_k\, (\forall \zeta \in T_{y_k}Y_{T_k})$ for any $k<j$, $k \ne k_0$ (hence (B) holds).
To prove the claim, we proceed as follows: 

\textbf{Step 1:}
Since $R^0_{w_jv_j}(r/2,r] \cap C_k = \emptyset$ for $k_0<k<j$, 
(a) and Remark \ref{rem:iota} show that $\varphi_k$ is identity (and isometry) on 
$R^0_{w_jv_j}(r/2,r]$ for any $k_0<k<j$. 
This proves the claim for $k_0<k<j$. 

\textbf{Step 2:}
Since $y_{k_0} = y_{j-1} \in R^0_{w_jv_j}(r/2,r] \subset R^1_{v_{k_0}w_{k_0}}(r/2,r]$, 
(d) implies that $y_{k_0-1} \in A_{k_0} \setminus B_{w_{k_0}}(r)$. 
Hence (i)-(ii) proves the claim for $k<k_0$.

\end{enumerate}

Finally, we consider the case (ii). In this case, $y_k \notin C_k$ for any $k>j$.
Hence (a) and Remark \ref{rem:iota} show 
$|d\varphi_k(\zeta)|_{k-1}=|\zeta|_k\,(\forall \zeta \in T_{y_k}Y_{T_k})$ for any $k>j$. Hence (A) holds.
\end{proof}

Finally we prove Lemma \ref{lem:W}.
Take $\varphi_2, \ldots, \varphi_m$ as in Lemma \ref{lem:psiphi}, and let 
$\varphi'':= \varphi_2 \circ \cdots \circ \varphi_m: Y_{T_m} \to Y_{T_1}$. 
Then, Lemma \ref{lem:C53} implies that for any $y \in Y_T$, 
the number of $2 \le j \le m$ such that there exists $\xi \in T_{y_j} Y_{T_j}$ satisfying 
$|d\varphi_j(\xi)|_{j-1} >|\xi|_j$ is at most $3$. 
Hence (c) shows that $d^+(\varphi'': g_m, g_1) \le (c_5)^3$. 

On the other hand, $T_1 = \{v_1\}$, $Y_{T_1}=T(r) \setminus B_{v_1}(r/2)$.
Hence there exists a diffeomorphism 
$\varphi''': Y_{T_1} \to T(r) \setminus \{i(v_1)\}$ such that 
$d^+(\varphi''': g_1, g_{T(r)})\le \const_n$. 
Hence $\varphi':= \varphi''' \circ \varphi''$ satisfies
$d^+(\varphi': g_{X_T}, g_{T(r)}) \le \const_n$
(recall that 
$g_m = g_{X_T}|_{Y_{T_m}}$).

\textbf{Acknowledgments.}

The author would like to appreciate 
Professor Kenji Fukaya for his warm encouragement and precious comments on the preliminary version of this paper. 
He also thanks the referee for many useful comments. 
The most part of this research was conducted when the author was supported by JSPS KAKENHI Grant Number 11J01157.
The author is currently supported by JSPS KAKENHI Grant Number 25800041.

\end{document}